\documentclass[10pt,reqno]{amsart}

\usepackage{fixltx2e}                       
\usepackage[usenames,dvipsnames]{xcolor}
\usepackage{fancyhdr}
\usepackage{amsmath,amsfonts,amsbsy,amsgen,amscd,mathrsfs,amssymb,amsthm}
\usepackage{mathtools}
\usepackage[font=small,margin=0.25in,labelfont={sc},labelsep={colon}]{caption}
\usepackage{tikz}
\usepackage{enumerate}
\usepackage{graphicx}

\usepackage{bm} 

\usepackage[colorlinks=true]{hyperref}

\newtheorem*{mthm}{Theorem}
\newtheorem{thm}{Theorem}[section]
\newtheorem{lem}[thm]{Lemma}

\newtheorem{prop}[thm]{Proposition}
\newtheorem{cor}[thm]{Corollary}

\theoremstyle{definition}

\newtheorem{defn}[thm]{Definition}
\newtheorem{example}[thm]{Example}
\newtheorem{rem}[thm]{Remark}

\numberwithin{equation}{section} 
\numberwithin{figure}{section}
\numberwithin{table}{section}

\begin{document}

\title{The Borell-Ehrhard Game}
\author{Ramon van Handel}
\address{Sherrerd Hall Room 227, Princeton University, Princeton, NJ 
08544, USA}
\thanks{Supported in part by NSF grant
CAREER-DMS-1148711 and by the ARO through PECASE award 
W911NF-14-1-0094.}
\email{rvan@princeton.edu}

\begin{abstract}
A precise description of the convexity of Gaussian measures is provided by 
sharp Brunn-Minkowski type inequalities due to Ehrhard and Borell. We show 
that these are manifestations of a game-theoretic mechanism: a minimax 
variational principle for Brownian motion. As an application, we obtain a 
Gaussian improvement of Barthe's reverse Brascamp-Lieb inequality.
\end{abstract}



\maketitle

\thispagestyle{empty}

\section{Introduction}

The convexity properties of probability measures play an important role in 
various areas of probability theory, analysis, and geometry. They arise in 
a fundmental manner, for example, in the study of concentration phenomena 
\cite{Led01,BGL14} and in functional analysis and convex geometry 
\cite{BGVV14,AGM15}. Among the most delicate results in this area are the 
remarkable convexity properties of Gaussian measures 
\cite{Ehr83,Lif95,Lat03,Bor03,Bor08}.
The aim of this paper is to shed some new light on the latter topic.

Let $\gamma_n$ be the standard Gaussian measure on
$\mathbb{R}^n$. The simplest expression of the convexity of 
Gaussian measures is given by the log-concavity property:
$$
	\lambda\log(\gamma_n(A)) +
	(1-\lambda)\log(\gamma_n(B)) \le
	\log(\gamma_n(\lambda A+(1-\lambda)B))
$$
for all $\lambda\in[0,1]$ and Borel sets $A,B\subseteq\mathbb{R}^n$, where 
$A+B:=\{x+y:x\in A,y\in B\}$ denotes Minkowski addition. This inequality 
is easily deduced from the classical Brunn-Minkowski inequality, which is 
the analogous statement for Lebesgue measure. However, while the 
importance of log-concavity can hardly be overstated, we expect in the 
case of Gaussian measures that convexity should appear in a much stronger 
form than can be explained by log-concavity alone. For example, the 
classical isoperimetric inequality for Euclidean volume is an easy and 
fundamental consequence of the Brunn-Minkowski inequality \cite{Gar02}, 
but log-concavity fails to explain the analogous isoperimetric property of 
Gaussian measures \cite{Lat03}.

A precise description of the convexity of Gaussian measures
was developed in a remarkable paper by Ehrhard \cite{Ehr83}, who
introduced the following sharp analogue of the Brunn-Minkowski inequality
for Gaussian measures:
$$
	\lambda\,\Phi^{-1}(\gamma_n(A))+
	(1-\lambda)\,\Phi^{-1}(\gamma_n(B)) \le
	\Phi^{-1}(\gamma_n(\lambda A+(1-\lambda)B)),
$$
where $\Phi(x):=\gamma_1((-\infty,x])$. This inequality becomes equality 
when $A,B$ are parallel halfspaces, and is a strict improvement over 
log-concavity as the function $\log\Phi$ is concave. It has numerous 
interesting and important implications, including the isoperimetric 
property of Gaussian measures that arises as a special case 
\cite{Lat03,Lif95}.

Given the fundamental nature of Ehrhard's inequality, it is natural to 
seek other Gaussian analogues of the rich family of results that appear in 
the classical Brunn-Minkowski theory (cf.\ \cite{Gar02,Bar06} and the 
references therein). Progress in this direction has remained relatively 
limited, however. Unlike the classical Brunn-Minkowski inequality, which 
is well understood from many different perspectives, only two approaches 
to Ehrhard's inequality are known. Ehrhard's original proof \cite{Ehr83}, 
using a Gaussian analogue of Steiner symmetrization, is limited to the 
case where the sets $A,B$ are convex; it was later extended by Lata{\l}a 
\cite{Lat96} to eliminate the convexity assumption on one of the two sets. 
The long-standing problem of proving Ehrhard's inequality for arbitrary 
Borel sets was finally settled by Borell \cite{Bor03}, who also introduced 
a number of significant generalizations of this inequality \cite{Bor07, 
Bor08}. Borell's elegant approach, using a nonlinear heat equation and the 
parabolic maximum principle, relies on some delicate cancellations (as 
will be explained below), complicating efforts to identify how it can be 
applied in other settings. A more abstract variant of Borell's approach is 
given in \cite{BH09,IV15}, but the mechanism that makes this approach work 
remains somewhat mysterious. Let us note, in addition, that unlike many 
other geometric inequalities (including the Gaussian isoperimetric 
inequality) that extend to more general settings, Ehrhard's inequality 
appears to be uniquely Gaussian; see \cite[\S 4.3]{KM16} for some 
discussion on this point.

The aim of this paper is to develop a new interpretation of Ehrhard's 
inequality: we will show that both Ehrhard's inequality and its 
generalizations due to Borell arise as manifestations of a stochastic game 
that appears to lie at the heart of these phenomena. This unexpected 
game-theoretic mechanism provides new insight into the success of earlier
proofs, and allows us to identify new convexity results for Gaussian 
measures. In particular, we will develop a Gaussian improvement of 
Barthe's reverse Brascamp-Lieb inequality, addressing a question posed in 
\cite{BH09}.

\subsection{Borell's stochastic method}

To motivate the ideas that will be introduced in the sequel, let us begin 
by recalling a powerful approach, also due to Borell \cite{Bor00}, for 
proving log-concavity of Gaussian measures.

In order to show that $\gamma_n$ (or any other measure) is log-concave,
it is natural to seek a representation formula for $\log(\gamma_n(A))$ 
from which the concavity property becomes evident. A fundamental 
representation of this type, the Gibbs variational principle, 
dates back to the earliest work on statistical mechanics \cite{Gib02}:
$$
	\log\bigg(
	\int e^f\,d\gamma_n
	\bigg) =
	\sup_{\mu}\bigg\{
	\int f\,d\mu - H(\mu||\gamma_n)
	\bigg\},
$$
where $H(\mu||\gamma_n)$ denotes relative entropy and the supremum is
taken over all probability measures $\mu$. The log-concavity property 
could be read off directly from this formulation using displacement 
convexity of relative entropy as developed in the theory of optimal 
transportation \cite{Vil03}. However, in the case of Gaussian measures,
a simpler approach becomes available by identifying $\gamma_n$ with the 
distribution of the value of a Brownian motion $\{W_t\}$ at time one.
The advantage gained by this approach is that absolutely continuous
changes of measure of Brownian motion admit an explicit characterization
by Girsanov's theorem \cite{LS01}, which gives rise to the following 
reformulation of the Gibbs variational principle for Gaussian measures:
$$
        \log\bigg(   
        \int e^f\,d\gamma_n
        \bigg) =
	\sup_\alpha
	\mathbf{E}\bigg[
	f\bigg(
	W_1+\int_0^1 \alpha_t\,dt
	\bigg) - \frac{1}{2}\int_0^1 \|\alpha_t\|^2\,dt
	\bigg],
$$
where the supremum is taken over all progressively measurable processes
$\alpha$. This formula was originally obtained using PDE 
methods by Fleming \cite{FS06}; the connection with the Gibbs
variational principle was developed by Bou\'e and Dupuis \cite{BD98}.

It was observed by Borell in \cite{Bor00} that log-concavity of the 
Gaussian measure is an almost immediate consequence of this identity. Let 
us illustrate this idea in its functional (Pr\'ekopa-Leindler) form.
Let $f,g,h$ be functions such that
$$
	\lambda\log(f(x)) + (1-\lambda)\log(g(y)) \le
	\log(h(\lambda x+(1-\lambda)y))
$$
for all $x,y$, and denote by $\alpha^f$ and $\alpha^g$ the maximizing 
processes when the above representation is applied to $\log f$ and $\log 
g$, respectively. Then we have
\begin{align*}
	& \lambda\log\bigg(\int f\,d\gamma_n\bigg) +
	(1-\lambda)\log\bigg(\int g\,d\gamma_n\bigg) \\
	&=
	\lambda\,
        \mathbf{E}\bigg[
        \log f\bigg(
        W_1+\int_0^1 \alpha^f_t\,dt  
        \bigg) 
	- \frac{1}{2}\int_0^1 \|\alpha_t^f\|^2\,dt
        \bigg] \\
	&\qquad + (1-\lambda)\,
        \mathbf{E}\bigg[
        \log g\bigg(
        W_1+\int_0^1 \alpha^g_t\,dt  
        \bigg) 
	- \frac{1}{2}\int_0^1 \|\alpha_t^g\|^2\,dt
        \bigg] \\
	& \le
        \mathbf{E}\bigg[
        \log h\bigg(
        W_1+\int_0^1 (\lambda\alpha^f_t+(1-\lambda)\alpha^g_t)\,dt  
        \bigg) 
	- \frac{1}{2}\int_0^1 \|\lambda\alpha_t^f+(1-\lambda)\alpha^g_t\|^2\,dt
        \bigg] \\
	&\le
	\log\bigg(\int h\,d\gamma_n\bigg).
\end{align*}
Log-concavity follows readily by choosing $f=\mathbf{1}_A$,
$g=\mathbf{1}_B$, and $h=\mathbf{1}_{\lambda A+(1-\lambda)B}$.
The beauty of this stochastic approach is that it reduces log-concavity of 
Gaussian measures to a trivial fact, viz.\ convexity of the function 
$x\mapsto\|x\|^2$. This idea has been further developed in
\cite{Leh13, Leh14, CM15} to prove various other inequalities, some of 
which do not seem to be readily accessible by other methods.

It is tempting to approach Ehrhard's inequality by seeking a Gaussian 
improvement of the Gibbs variational principle. It is far from clear, 
however, why this should be possible. The Gibbs variational principle is 
not a mysterious result: it simply expresses Fenchel duality for the 
convex functional $f\mapsto \log(\int e^f d\gamma_n)$. On the other hand, 
classical results of Hardy, Littlewood, and P\'olya \cite[\S 3.16]{HLP88} 
imply that the functional $f\mapsto \Phi^{-1}(\int\Phi(f)\,d\gamma_n)$ 
cannot be convex.

In his proof of Ehrhard's inequality \cite{Bor03}, Borell circumvents the 
lack of a representation formula by using partial differential equation 
methods. As a first step, he obtains a PDE for the transformation 
$v_f(t,x):=\Phi^{-1}(u_f(t,x))$ of the solution $u_f(t,x)$ of the heat 
equation with initial condition $f$ (the latter arises naturally in this 
setting as the Markov semigroup of Brownian motion). It is not immediately 
obvious that the resulting nonlinear PDE, given in section 
\ref{sec:borpde} below, possesses any useful convexity properties. 
Instead, Borell considers directly the desired combination 
$C(t,x,y):=\lambda v_f(t,x)+(1-\lambda)v_g(t,y)-v_h(t,\lambda 
x+(1-\lambda)y)$, and observes that a fortuitous cancellation occurs: one 
can arrange the terms in the combined PDEs for $v_f,v_g,v_h$ to obtain a 
parabolic PDE for $C$ alone. This makes it possible to apply the parabolic 
maximum principle to deduce nonpositivity of $C$, which is essentially the 
statement of Ehrhard's inequality in its functional form.

\subsection{The Borell-Ehrhard game}

The main result of the present paper is a new stochastic representation 
formula that lies at the heart of Ehrhard's inequality, in direct 
analogy with Borell's stochastic approach to log-concavity. This principle 
provides significant insight into the mechanism behind the convexity 
properties of Gaussian measures, as well as a new tool to study such 
properties.

As was explained above, the lack of convexity of $f\mapsto 
\Phi^{-1}(\int\Phi(f)\,d\gamma_n)$ prohibits us from obtaining a 
representation formula by a convex duality argument. Instead, our main 
result shows that this functional can be represented by a minimax 
variational principle. An informal statement of our main result is as 
follows.

\begin{mthm}[informal statement]
For bounded and uniformly continuous $f$
\begin{align*}
        &\Phi^{-1}\bigg(
        \int \Phi(f)\,d\gamma_n
        \bigg) = \\
	&\qquad
	\sup_\alpha\inf_\beta
	\mathbf{E}\bigg[
	\int_0^1 e^{-\frac{1}{2}\int_0^t\|\beta_s\|^2ds}
	\langle \alpha_t,\beta_t\rangle\,dt +
	e^{-\frac{1}{2}\int_0^1\|\beta_t\|^2dt}
	f\bigg(
	W_1 + \int_0^1 \alpha_t\,dt
	\bigg)
	\bigg].
\end{align*}
\end{mthm}

This expression can be interpreted as the value of a zero-sum stochastic 
game between two players. The first player can apply a force $\alpha_t$ at 
time $t$ to the underlying Brownian motion. The second player cannot 
affect the dynamics of the Brownian motion, but can instead choose to end 
the game prematurely: her control $\beta_t$ is the rate of termination of 
the game at time $t$ (that is, the game ends prematurely in the interval 
$[t,t+dt)$ with probability $\|\beta_t\|^2dt$). The remarkable feature of 
this game is that the running cost $\langle \alpha_t,\beta_t\rangle$ is 
not quadratic, as in the stochastic representation used to prove 
log-concavity, but rather \emph{linear} in $\alpha$. This reflects the 
fact that the $\Phi^{-1}$ transformation lies precisely at the border of 
where we can expect convexity to appear: it ``linearizes'' the quadratic 
cost that arises from the Gibbs variational principle. (It is pointed out 
in \cite{CM15} that log-concavity of $\gamma_n$ can be strengthened in a 
different sense by exploiting uniform convexity of the quadratic cost.)

As is typical in the theory of continuous-time games, it is essential to 
carefully define the information structure available to each player in 
order for the above stochastic representation to be valid. In section 
\ref{sec:begame}, we provide a precise formulation and proof of our main 
result. The essential observation behind the proof is that our stochastic 
game is closely connected to Borell's PDE approach to Ehrhard's 
inequality: the nonlinear heat equation of Borell can be identified as the 
Bellman-Isaacs equation \cite{Sw96, FS89} for the value of our stochastic 
game. This observation leads not only to the above representation, but 
also reveals the reason behind the hidden convexity that appears somewhat 
mysteriously in Borell's proof.

With the above stochastic representation in hand, it is a simple exercise 
to deduce Ehrhard's inequality, and its generalizations due to Borell, in 
complete analogy to the stochastic proof of log-concavity. This exercise 
is carried out in section \ref{sec:app}.

\subsection{A Gaussian reverse Brascamp-Lieb inequality}

As an illustration of the power of the stochastic approach, we will use it 
to obtain a Gaussian improvement of the reverse Brascamp-Lieb inequality 
of Barthe \cite{Bar98,BH09}. Let us first recall Barthe's inequality in 
its Brunn-Minkowski form (see section \ref{sec:bl} for the functional
form). Let 
$E_1,\ldots,E_k$ be linear subspaces of $\mathbb{R}^n$ with 
$\mathrm{dim}(E_i)=n_i$. Denote by $P_i$ the orthogonal projection on 
$E_i$, and let $\lambda_1,\ldots,\lambda_k\ge 0$ be such that 
$\lambda_1P_1+\cdots+\lambda_kP_k=I_n$. Then Barthe's inequality states 
that for any Borel sets $A_i\subseteq E_i$, we have
$$
	\lambda_1\log(\gamma_{n_1}(A_1)) +
	\cdots + \lambda_k\log(\gamma_{n_k}(A_k)) \le
	\log(\gamma_n(\lambda_1A_1+\cdots+\lambda_kA_k)),
$$
where we identify $\gamma_{n_i}$ with the standard Gaussian measure on 
$E_i$. This is an extension of the log-concavity property where the sets 
$A_i$ may lie in lower-dimensional subspaces of the ambient space (in 
which case log-concavity is a trivial statement).

In view of Ehrhard's inequality, one might hope that it is possible to 
replace the logarithm by $\Phi^{-1}$ in Barthe's inequality to obtain a 
Gaussian improvement. However, this is certainly impossible in general: if 
$E_1,\ldots,E_k$ are orthogonal subspaces that span $\mathbb{R}^n$, then 
Barthe's inequality is in fact equality for all choices of $A_i$ and no 
Gaussian improvement is possible (see section \ref{sec:barthe}). 
Nonetheless, it is possible to systematically improve Barthe's inequality 
in the Gaussian setting, as we will do in section \ref{sec:bl}. To this 
end, define for every $c\in(0,1)$ the function
$$
	\Phi_c^{-1}(x) := \Phi^{-1}(cx)-\Phi^{-1}(c).
$$
We will show in section \ref{sec:bl} that, under the same 
assumptions as in Barthe's inequality,
$$
	\lambda_1\Phi_c^{-1}(\gamma_{n_1}(A_1)) +
	\cdots + \lambda_k\Phi_c^{-1}(\gamma_{n_k}(A_k)) \le
	\Phi_c^{-1}(\gamma_n(\lambda_1A_1+\cdots+\lambda_kA_k))
$$
for every $c\in(0,1)$. From this inequality, one can recover both Barthe's 
inequality ($c\downarrow 0$) and Ehrhard's inequality ($c\uparrow 1$) as
special cases. The stochastic game approach was essential to discovering 
the correct formulation of this inequality.

\subsection{Generalized means}

While our main result sheds new light on the mechanism behind Ehrhard's 
inequality, there remains some residual mystery regarding the origin of 
this stochastic game. The stochastic representation used to prove 
log-concavity is entirely natural, as it arises simply as a specialization 
of the Gibbs variational principle to the Brownian setting. It is unclear, 
however, whether there exists a natural minimax generalization of the 
Gibbs variational principle that can provide an analogous explanation for 
the Borell-Ehrhard game.

From another perspective, however, there is nothing particularly 
surprising about the stochastic representations that we encountered so 
far. To place these results in a broader context, we can consider the more 
general functional $f\mapsto F^{-1}(\int F(f)\,d\gamma_n)$ for any 
strictly increasing function $F$. Such functionals, called generalized 
means, were studied by Hardy, Littlewood, and P\'olya \cite[chapter 
3]{HLP88}, who provide in particular necessary and sufficient conditions 
for such functionals to be convex. In section \ref{sec:hlp}, we will show 
that any convex generalized mean admits a stochastic representation that 
is very similar to the special case $F(x)=e^x$, from which convexity can 
be immediately read off. Moreover, we will argue that essentially 
arbitrary choice of $F$ will admit a stochastic game representation, so 
the appearance of a game in the case $F(x)=\Phi(x)$ is just one specific 
example. Of course, it is a special feature of this example that gave rise 
to Ehrhard's inequality; the potential utility of such representations in 
other contexts will depend on the problem at hand.

\section{The Borell-Ehrhard game}
\label{sec:begame}

\subsection{Setting and main result}

Let $(\Omega,\mathcal{F},\{\mathcal{F}_t\},\mathbf{P})$ be a probability 
space with a complete and right-continuous filtration, and let $\{W_t\}$ 
be a standard $n$-dimensional $\mathcal{F}_t$-Brownian motion. We denote 
by $\gamma_n$ the standard Gaussian measure on $\mathbb{R}^n$ and by 
$\Phi(x):=\gamma_1((-\infty,x])$. Our main result is a variational 
principle for Gaussian measures that will be expressed as a stochastic 
game for the Brownian motion $W$.

As is often the case in continuous time games, it is important to 
carefully define what information is available to each player. 
Informally, we can view our game as the continuous time limit of a 
discrete time game where two players take turns exercising some control on 
the underlying Brownian motion. We denote the controls of the first and 
second players at time $t$ by $\alpha_t$ and $\beta_t$, respectively. As 
the second player comes after the first, her control may depend on the 
choice of control of the first player. Conversely, the control of the 
first player may depend on the choice of control of the second player in 
earlier turns. It is not entirely obvious how this information structure 
should be encoded when time is continuous.

For our purposes, it will be convenient to adopt an approach due to
Elliott and Kalton \cite{EK72,FS89}. In this framework, the second player 
may choose any control.

\begin{defn}
A \emph{control} is a progressively measurable $n$-dimensional
process $\beta=\{\beta_t\}_{t\in[0,1]}$. 
Denote by $\mathcal{C}$ the family of all controls such that 
$\mathbf{E}[\int_0^1\|\beta_s\|^2ds]<\infty$. 
\end{defn}

On the other hand, the action of the first player must explicitly account 
for the fact that she has access to the earlier choice of control of the 
second player. To this end, we introduce the notion of an (Elliott-Kalton) 
strategy.

\begin{defn}
A \emph{strategy} is a map $\alpha:\mathcal{C}\to\mathcal{C}$ such that
for every $t\in[0,1]$ and $\beta,\beta'\in\mathcal{C}$ such that
$\beta_s(\omega)=\beta_s'(\omega)$ for a.e.\ 
$(s,\omega)\in[0,t]\times\Omega$, we have
$\alpha_s(\beta)(\omega)=\alpha_s(\beta')(\omega)$ for a.e.\ 
$(s,\omega)\in[0,t]\times\Omega$.
Denote by $\mathcal{S}$ the family of all strategies such that
$\sup\{\mathbf{E}[\int_0^1\|\alpha_s(\beta)\|^2ds]:
\mathbf{E}[\int_0^1\|\beta_s\|^2ds]\le R\}<\infty$
for all $R<\infty$.
\end{defn}

In the Elliott-Kalton approach, the second player chooses any control
$\beta\in\mathcal{C}$, while the first player's control $\alpha(\beta)$ 
is defined by a strategy $\alpha\in\mathcal{S}$. The definition of a 
strategy ensures that the control of the first player depends causally on 
the control of the second player, thereby encoding the desired
information structure.

With these formalities out of the way, we can now formulate our main 
result.

\begin{thm}
\label{thm:main}
Let $f:\mathbb{R}^n\to\mathbb{R}$ be bounded and uniformly continuous,
and define
$$
	J_f[\alpha,\beta] :=
	\mathbf{E}\bigg[
	\int_0^1 e^{-\frac{1}{2}\int_0^t\|\beta_s\|^2ds}
	\langle \alpha_t,\beta_t\rangle\,dt +
	e^{-\frac{1}{2}\int_0^1\|\beta_t\|^2dt}
	f\bigg(
	W_1 + \int_0^1 \alpha_t\,dt
	\bigg)
	\bigg]
$$
for $\alpha,\beta\in\mathcal{C}$. Then
$$
	\Phi^{-1}\bigg(
	\int \Phi(f)\,d\gamma_n
	\bigg) 
	=
	\sup_{\alpha\in\mathcal{S}}\inf_{\beta\in\mathcal{C}}
	J_f[\alpha(\beta),\beta]
	=
	\inf_{\alpha\in\mathcal{S}}\sup_{\beta\in\mathcal{C}}
	J_f[\alpha(\beta),\beta].
$$
\end{thm}

The remainder of this section is devoted to the proof of Theorem 
\ref{thm:main}. The connection with geometric inequalities
will be developed in sections \ref{sec:app} and \ref{sec:bl} below.

\subsection{The Borell PDE}
\label{sec:borpde}

Throughout the proof, we will assume without loss of generality that $f$ 
is bounded, smooth, and has bounded derivatives of all orders. Once the 
result is proved in this case, the conclusion is readily extended to 
functions $f$ that are only bounded and uniformly continuous (as the 
latter can be approximated in the uniform topology by smooth functions 
with bounded derivatives by convolution with a smooth compactly supported 
kernel, cf.\ \cite[\S 8.2]{Fol99}).

Define for $(t,x)\in[0,1]\times\mathbb{R}^n$ the function
$$
	u(t,x) := \mathbf{E}[\Phi(f(W_1-W_t+x))],
$$
so that $u$ solves the heat equation
$$
	\frac{\partial u}{\partial t} + \frac{1}{2}\Delta u=0,
	\qquad
	u(1,x) = \Phi(f(x)).
$$
Define
$$
	v(t,x) := \Phi^{-1}(u(t,x)).
$$
By the smoothness assumption on $f$ and elementary properties of the heat 
equation, $u$ and therefore $v$ are bounded, smooth, and have bounded 
derivatives of all orders on $[0,1]\times\mathbb{R}^n$. Moreover, it is 
readily verified that $v$ satisfies
$$
	\frac{\partial v}{\partial t}+\frac{1}{2}\Delta v -
	\frac{1}{2}v\|\nabla v\|^2=0,\qquad
	v(1,x)=f(x).
$$
This equation was introduced by Borell \cite{Bor03} in his study
of the Ehrhard inequality.

The following simple observation contains the main idea behind the proof 
of Theorem \ref{thm:main}: the nonlinear term in Borell's PDE admits a 
variational interpretation.

\begin{lem}
\label{lem:isaacs}
Let $c$ be a constant such that $2c\ge\sup_xf(x)$. Then
$$
	-\frac{1}{2}v\|\nabla v\|^2 = 
	\sup_{a\in\mathbb{R}^n}\inf_{b\in\mathbb{R}^n}
	\bigg\{
	\langle a+cb,\nabla v+b\rangle
	-\frac{1}{2}v\|b\|^2
	\bigg\},
$$
where the optimizer $a^*=(c-v)\nabla v$, $b^*=-\nabla v$
is a saddle point.
\end{lem}

\begin{proof}
Define for $a,b\in\mathbb{R}^n$ the objective
$$
	H(a,b):=
	\langle a+cb,\nabla v+b\rangle
        -\frac{1}{2}v\|b\|^2.
$$
Then it is readily verified that
$$
	H(a,b^*) = 
	-\frac{1}{2}v\|\nabla v\|^2,
	\qquad
	H(a^*,b) =
	\frac{1}{2}(2c-v)\|b+\nabla v\|^2
        -\frac{1}{2}v\|\nabla v\|^2.
$$
But note that as $2c\ge f$, we have $2c-v\ge 0$ by the definition 
of $v$. Therefore
$$
	\sup_a\inf_b H(a,b) \le
	\sup_a H(a,b^*) =
	-\frac{1}{2}v\|\nabla v\|^2 =
	\inf_b H(a^*,b)
	\le \sup_a\inf_b H(a,b),
$$
and the proof is complete.
\end{proof}

Lemma \ref{lem:isaacs} reveals that the partial differential equation 
satisfied by $v$ is none other than the Bellman-Isaacs equation for the 
value of a stochastic game \cite{FS89,Sw96}. We can now proceed along 
mostly standard lines to formalize this idea.

\subsection{Upper bound}

Fix $2c\ge f$, and consider the stochastic differential equation
$$
	dX_t^\beta = 
	(c-v(t,X_t^\beta))\nabla v(t,X_t^\beta)\,dt + 
	c\beta_t\,dt+dW_t,
	\qquad
	X_0^\beta=0
$$
for $\beta\in\mathcal{C}$.
As the function $(c-v)\nabla v$ is smooth with bounded derivatives, this
equation has a unique strong solution $X^\beta$ \cite[Theorem 4.8]{LS01}. 
Define
$$
	\alpha^*_t(\beta):= (c-v(t,X_t^\beta))\nabla v(t,X_t^\beta).
$$
Then evidently $\alpha^*(\beta)\in\mathcal{C}$ (in fact, it is 
uniformly bounded) and $\alpha^*$ depends causally on $\beta$. Thus 
we have shown that $\alpha^*\in\mathcal{S}$ defines an 
Elliott-Kalton strategy.

Applying It\^o's formula to the process 
$e^{-\frac{1}{2}\int_0^t\|\beta_s\|^2ds}v(t,X_t^\beta)$ gives
\begin{align*}
	&\int_0^1e^{-\frac{1}{2}\int_0^t\|\beta_s\|^2ds}
	\langle \alpha^*_t(\beta)+c\beta_t,\beta_t\rangle\,dt+
	e^{-\frac{1}{2}\int_0^1\|\beta_t\|^2dt}f(X_1^\beta) = \\
	&v(0,0) 
	+ \int_0^1 e^{-\frac{1}{2}\int_0^t\|\beta_s\|^2ds}
	\langle \nabla v(t,X_t^\beta),dW_t\rangle  \\
	& + \int_0^1 e^{-\frac{1}{2}\int_0^t\|\beta_s\|^2ds}
	\bigg\{
	\frac{\partial v}{\partial t}(t,X_t^\beta)
	+\frac{1}{2}\Delta v(t,X_t^\beta) \\
	& \phantom{+ \int_0^1 e^{-\frac{1}{2}\int_0^t\|\beta_s\|^2ds}\bigg\{~}
	+\langle \alpha^*_t(\beta)+c\beta_t,
	\nabla v(t,X_t^\beta)+\beta_t\rangle
	-\frac{1}{2}v(t,X_t^\beta)\|\beta_t\|^2
	\bigg\}dt.
\end{align*}
We now observe that the last integral in this expression is nonnegative
by Borell's PDE and Lemma \ref{lem:isaacs}. Moreover, the
Brownian integral is a martingale as $\nabla v$ is bounded. 
Therefore, taking the expectation of this expression, we obtain
\begin{align*}
	v(0,0)
	&\le
	\mathbf{E}\bigg[\int_0^1e^{-\frac{1}{2}\int_0^t\|\beta_s\|^2ds}
	\langle \alpha^*_t(\beta)+c\beta_t,\beta_t\rangle\,dt+
	e^{-\frac{1}{2}\int_0^1\|\beta_t\|^2dt}f(X_1^\beta)
	\bigg] \\
	& = 
	J_f[\alpha^*(\beta)+c\beta,\beta]
\end{align*}
for every $\beta\in\mathcal{C}$. But evidently 
$\tilde\alpha^*(\beta):=\alpha^*(\beta)+c\beta$
defines another Elliott-Kalton strategy $\tilde\alpha^*\in\mathcal{S}$.
We therefore readily obtain the upper bound in Theorem \ref{thm:main}
$$
	\Phi^{-1}\bigg(
	\int \Phi(f)\,d\gamma_n
	\bigg) = v(0,0) 
	\le \sup_{\alpha\in\mathcal{S}}\inf_{\beta\in\mathcal{C}}
	J_f[\alpha(\beta),\beta].
$$

\subsection{Lower bound}

For the proof of the lower bound, fix any $\alpha\in\mathcal{S}$. Given 
this strategy, our aim is to construct a control $\beta\in\mathcal{C}$ 
that nearly minimizes $J_f[\alpha(\beta),\beta]$. We will do this by 
imitating the idea that our continuous game is the limit of discrete-time 
games, as was explained informally at the beginning of this section.

To this end, fix a time step $\delta=N^{-1}$ ($N\ge 1$).
For $t\in[0,\delta)$, let $\beta_t:=-\nabla v(0,0)$. We now iteratively
extend the definition of $\beta$ as follows. Suppose that $\beta$ has been 
defined on the interval $[0,k\delta)$. Then $\alpha_t(\beta)$ is 
uniquely defined for a.e.\ $t\in[0,k\delta)$ (as $\alpha$ is causal
by the definition of an Elliott-Kalton strategy). Writing
$$
	X_t := W_t + \int_0^t \alpha_s(\beta)\,ds,
$$
we define $\beta_t:=-\nabla v(k\delta,X_{k\delta})$ for 
$t\in[k\delta,(k+1)\delta)$. Iterating this process $N-1$ times results
in a control $\beta\in\mathcal{C}$ that is uniquely defined a.e.\ in
$[0,1]\times\Omega$.

Applying It\^o's formula to
$e^{-\frac{1}{2}\int_0^t\|\beta_s\|^2ds}v(t,X_t)$ as in the upper bound 
gives
$$
	J_f[\alpha(\beta),\beta] = v(0,0) + \mathbf{E}[\Gamma]
$$
where
$$
	\Gamma := 
	\int_0^1 e^{-\frac{1}{2}\int_0^t\|\beta_s\|^2ds}
	\bigg\{
	\frac{1}{2}v(t,X_t)(\|\nabla v(t,X_t)\|^2-\|\beta_t\|^2)
	+\langle \alpha_t(\beta),
	\nabla v(t,X_t)+\beta_t\rangle
	\bigg\}dt.
$$
As $v$ is bounded and has bounded derivatives of all orders, we can 
estimate
\begin{align*}
	\Gamma 
	&\le
	C_1\int_0^1 
	(1+\|\alpha_t(\beta)\|) \|\nabla v(t,X_t)+\beta_t\|
	\,dt \\
	&=
	C_1\sum_{k=0}^{N-1}\int_{k\delta}^{(k+1)\delta} 
	(1+\|\alpha_t(\beta)\|) \|\nabla v(t,X_t)-
		\nabla v(k\delta,X_{k\delta})\|
	\,dt
	\\
	&\le
	C_2\sum_{k=0}^{N-1}\int_{k\delta}^{(k+1)\delta} 
	(1+\|\alpha_t(\beta)\|)(\delta+\|X_t-X_{k\delta}\|)
	\,dt
\end{align*}
for constants $C_1,C_2$ that depend on $f$ only. Note that for 
$t\le(k+1)\delta$
$$
	\|X_t-X_{k\delta}\| \le
	\|W_t-W_{k\delta}\| + 
	\sqrt{\delta}
	\bigg[\int_{k\delta}^t \|\alpha_s(\beta)\|^2\,ds\bigg]^{1/2}.
$$
We can therefore estimate using Cauchy-Schwarz
$$
	\mathbf{E}[\Gamma] \le
	C_3\sqrt{\delta}
	\bigg(1+\mathbf{E}\bigg[\int_0^1\|\alpha_t(\beta)\|^2dt\bigg]\bigg)
	\le C_3(K+1)\sqrt{\delta},
$$
where 
$K:=\sup\{\mathbf{E}[\int_0^1\|\alpha_t(\beta')\|^2dt]:\|\beta'\|_\infty
\le\|\nabla v\|_\infty\}<\infty$ by definition as $\alpha\in\mathcal{S}$
and where $C_3$ depends only on $f$. We have therefore shown that
$$
	\inf_{\beta'\in\mathcal{C}}J_f[\alpha(\beta'),\beta']
	\le
	J_f[\alpha(\beta),\beta] \le v(0,0) +
	C_3(K+1)\sqrt{\delta}.
$$
As $\delta>0$ and $\alpha\in\mathcal{S}$ were arbitrary, we readily 
conclude that
$$
	\sup_{\alpha\in\mathcal{S}}\inf_{\beta\in\mathcal{C}}
	J_f[\alpha(\beta),\beta] \le v(0,0) = 
	\Phi^{-1}\bigg(
	\int \Phi(f)\,d\gamma_n
	\bigg).
$$

\subsection{End of proof}

Combining the upper and lower bound, we have shown
$$
	\Phi^{-1}\bigg(
        \int \Phi(f)\,d\gamma_n
        \bigg) =
	\sup_{\alpha\in\mathcal{S}}\inf_{\beta\in\mathcal{C}}
        J_f[\alpha(\beta),\beta].
$$
It remains to prove the second identity in Theorem \ref{thm:main}.
To this end, note that
$$
	\Phi^{-1}\bigg(
        \int \Phi(f)\,d\gamma_n\bigg) =
	-\Phi^{-1}\bigg(
        \int \Phi(-f)\,d\gamma_n\bigg) =
	-\sup_{\alpha\in\mathcal{S}}\inf_{\beta\in\mathcal{C}}
        J_{-f}[\alpha(\beta),\beta]
$$
as $\Phi(-x)=1-\Phi(x)$. But we can write
$$
	-\sup_{\alpha\in\mathcal{S}}\inf_{\beta\in\mathcal{C}}
        J_{-f}[\alpha(\beta),\beta]=
	\inf_{\alpha\in\mathcal{S}}\sup_{\beta\in\mathcal{C}}
        (-J_{-f}[\alpha(\beta),\beta])
	=
	\inf_{\alpha\in\mathcal{S}}\sup_{\beta\in\mathcal{C}}
        J_f[\alpha(\beta),-\beta].
$$
As $\mathcal{C}$ is invariant under the transformation $\beta\mapsto 
-\beta$ and $\mathcal{S}$ is invariant under the transformation
$\alpha(\beta)\mapsto \alpha(-\beta)$, the second identity in
Theorem \ref{thm:main} follows.

\section{The Ehrhard and Borell inequalities}
\label{sec:app}

The aim of this short section is to show that the classical Gaussian 
Brunn-Minkowski inequality of Ehrhard \cite{Ehr83,Bor03} and its 
generalizations due to Borell \cite{Bor07,Bor08} arise as immediate 
corollaries of Theorem \ref{thm:main}. In section \ref{sec:bl} below, we 
will extend this approach to derive new geometric inequalities for 
Gaussian measures.

\subsection{Ehrhard's inequality}

Ehrhard's inequality states that
$$
	\lambda\Phi^{-1}(\gamma_n(A)) 
	+ (1-\lambda)\Phi^{-1}(\gamma_n(B)) \le
	\Phi^{-1}(\gamma_n(\lambda A+(1-\lambda)B)) 
$$
for all Borel sets $A,B\subseteq\mathbb{R}^n$ and $\lambda\in[0,1]$. By 
approximating the indicator functions of $A$ and $B$ by smooth functions, 
it is routine to deduce this inequality from the following functional form 
of the result (see \cite{Bor03} or section \ref{sec:cornonsm} below).

\begin{cor}[\cite{Ehr83,Bor03}]
\label{cor:ehrhard}
Let $\lambda\in[0,1]$, and let $f,g,h$ be uniformly continuous functions 
with values in $[\varepsilon,1-\varepsilon]$ for some $\varepsilon>0$.
Suppose that for all $x,y\in\mathbb{R}^n$
$$
	\lambda\,\Phi^{-1}(f(x)) + (1-\lambda)\,\Phi^{-1}(g(y)) \le
	\Phi^{-1}(h(\lambda x+(1-\lambda)y)).
$$
Then
$$
	\lambda\,\Phi^{-1}\bigg(\int f\, d\gamma_n\bigg)+
	(1-\lambda)\,\Phi^{-1}\bigg(\int g\, d\gamma_n\bigg) \le
	\Phi^{-1}\bigg(\int h\, d\gamma_n\bigg).
$$
\end{cor}

\begin{proof}
Fix $\delta>0$, and choose near-optimal 
$\alpha_f,\alpha_g\in\mathcal{S}$ and $\beta_h\in\mathcal{C}$
such that
\begin{align*}
	\sup_{\alpha\in\mathcal{S}}\inf_{\beta\in\mathcal{C}}
	J_{\Phi^{-1}(f)}[\alpha(\beta),\beta] &\le
	\inf_{\beta\in\mathcal{C}}
        J_{\Phi^{-1}(f)}[\alpha_f(\beta),\beta] + \delta,
	\\
	\sup_{\alpha\in\mathcal{S}}\inf_{\beta\in\mathcal{C}}
	J_{\Phi^{-1}(g)}[\alpha(\beta),\beta] &\le
	\inf_{\beta\in\mathcal{C}}
        J_{\Phi^{-1}(g)}[\alpha_g(\beta),\beta] + \delta,
	\\
	J_{\Phi^{-1}(h)}[\lambda\alpha_f(\beta_h)+(1-\lambda)
	\alpha_g(\beta_h),\beta_h] &\le
	\inf_{\beta\in\mathcal{C}}
	J_{\Phi^{-1}(h)}[\lambda\alpha_f(\beta)+(1-\lambda)
        \alpha_g(\beta),\beta] + \delta.
\end{align*}
Then by Theorem \ref{thm:main}
\begin{align*}
	&\lambda\,\Phi^{-1}\big(\textstyle{\int f\, d\gamma_n}\big)+
	(1-\lambda)\,\Phi^{-1}\big(\textstyle{\int g\, d\gamma_n}\big) \\
	&\le \lambda\,
	J_{\Phi^{-1}(f)}[\alpha_f(\beta_h),\beta_h] +
	(1-\lambda)\,
	J_{\Phi^{-1}(g)}[\alpha_g(\beta_h),\beta_h] + 2\delta \\
	&\le
	J_{\Phi^{-1}(h)}[\lambda\alpha_f(\beta_h)+(1-\lambda)
	\alpha_g(\beta_h),\beta_h]
	+2\delta \\
	&\le
	\Phi^{-1}\big(\textstyle{\int h\, d\gamma_n}\big) + 3\delta,
\end{align*}
and the proof is completed by letting $\delta\downarrow 0$.
\end{proof}

\subsection{Borell's Gaussian Brunn-Minkowski inequalities}

In \cite{Bor07}, Borell proves a substantial generalization of Ehrhard's 
inequaliy: he shows that
$$
	\lambda\,\Phi^{-1}(\gamma_n(A)) 
	+ \mu\,\Phi^{-1}(\gamma_n(B)) \le
	\Phi^{-1}(\gamma_n(\lambda A+\mu B)) 
$$
holds for all Borel sets $A,B\subseteq\mathbb{R}^n$ if and only if
$\lambda+\mu\ge 1$ and $|\lambda-\mu|\le 1$ (the necessity of the latter 
conditions is easily verified by explicit examples, see \cite{Bor07}). 
The deduction of this result from Theorem \ref{thm:main} requires only a 
minor modification of the proof of Corollary \ref{cor:ehrhard}: it 
suffices to note that we do not need to choose the same Brownian motion 
$W$ in the variational problems for $f$ and $g$. By choosing instead 
two correlated Brownian motions, we immediately recover Borell's result.

\begin{cor}[\cite{Bor07}]
\label{cor:borell}
Let $\lambda,\mu\ge 0$, and let $f,g,h$ be uniformly continuous functions 
with values in $[\varepsilon,1-\varepsilon]$ for some $\varepsilon>0$.
Suppose that for all $x,y\in\mathbb{R}^n$
$$
	\lambda\,\Phi^{-1}(f(x)) + \mu\,\Phi^{-1}(g(y)) \le
	\Phi^{-1}(h(\lambda x+\mu y)).
$$
If $\lambda+\mu\ge 1$ and $|\lambda-\mu|\le 1$, then
$$
	\lambda\,\Phi^{-1}\bigg(\int f\, d\gamma_n\bigg)+
	\mu\,\Phi^{-1}\bigg(\int g\, d\gamma_n\bigg) \le
	\Phi^{-1}\bigg(\int h\, d\gamma_n\bigg).
$$
\end{cor}

\begin{proof}
Let $\rho = (1-\lambda^2-\mu^2)/2\lambda\mu$. The assumptions
$\lambda+\mu\ge 1$ and $|\lambda-\mu|\le 1$ guarantee that $\rho\in[-1,1]$.
We can therefore define two standard $n$-dimensional Brownian motions
$\{W_t\}$ and $\{\tilde W_t\}$ with quadratic covariation 
$\langle W^i,\tilde W^j\rangle_t = \rho t\delta_{ij}$. The point of this 
construction is that the process $\{\bar W_t\}$ defined as $\bar 
W_t:=\lambda W_t+\mu\tilde W_t$ is again a standard $n$-dimensional 
Brownian motion.

Let $\tilde J_f,\bar J_f$ be defined analogously to $J_f$ in Theorem 
\ref{thm:main} where $\{W_t\}$ is replaced by 
$\{\tilde W_t\}$ and $\{\bar W_t\}$, respectively. The remainder of the 
proof is identical to that of Corollary 
\ref{cor:ehrhard}, where $J_{\Phi^{-1}(g)}$ is replaced by $\tilde 
J_{\Phi^{-1}(g)}$ and $J_{\Phi^{-1}(h)}$ by $\bar 
J_{\Phi^{-1}(h)}$.
\end{proof}

\begin{rem}
We observe that it was essential for the success of the proof of Corollary 
\ref{cor:borell} that the game described by Theorem \ref{thm:main} is 
defined on a general probability space: while the objective function 
$J_f[\alpha,\beta]$ depends only on a single Brownian motion $\{W_t\}$, we 
allowed the controls $\alpha,\beta\in\mathcal{C}$ to be adapted to a 
larger filtration $\{\mathcal{F}_t\}$ that is not necessarily generated by 
the underlying Brownian motion alone. This freedom was used crucially in 
the proof of Corollary \ref{cor:borell}; here we can take $\mathcal{F}_t$ 
to be (the augmentation of) $\sigma\{W_s,\tilde W_s:s\le t\}$, but we 
cannot ensure that the control 
$\lambda\alpha_f(\beta_h)+\mu\alpha_g(\beta_h)$ will depend only on 
$\{\bar W_t\}$.
\end{rem}

The assumptions $\lambda+\mu\ge 1$ and $|\lambda-\mu|\le 1$ 
in Corollary \ref{cor:borell} are precisely the conditions
required for the existence of correlated standard Brownian motions
$\{W_{1,t}\}$ and $\{W_{2,t}\}$ such that $\lambda W_1+\mu W_2$ is also 
a standard Brownian motion. Along identical lines, we immediately see that 
the inequality 
$$
	\lambda_1\Phi^{-1}(\gamma_n(A_1)) + \cdots +
	\lambda_k\Phi^{-1}(\gamma_n(A_k)) 
	\le
	\Phi^{-1}(\gamma_n(\lambda_1A_1+\cdots+\lambda_kA_k)) 
$$
holds for all Borel sets $A_1,\ldots,A_k\subseteq\mathbb{R}^n$ whenever 
there exist correlated standard Brownian motions
$\{W_{i,t}\}$, $i=1,\ldots,k$ such that $\lambda_1W_1+\cdots+\lambda_kW_k$
is again a standard Brownian motion. The family of coefficients 
$\lambda_1,\ldots,\lambda_k\ge 0$ for which this is the case is  
characterized by \cite[Lemma 3]{BH09}, and we recover in this manner the 
general Gaussian Brunn-Minkowski inequality of Borell \cite{Bor08}.

\begin{rem}
We have stated Corollaries \ref{cor:ehrhard} and \ref{cor:borell}
for simplicity under the assumption that the functions $f,g,h$ are
uniformly continuous and bounded away from zero and one. This case 
contains the main difficulty of the problem: it is routine to derive from
this the corresponding results for sets \cite{Bor03}, and one can 
subsequently derive versions of Corollaries \ref{cor:ehrhard} and 
\ref{cor:borell} where the functions $f,g,h$ are just Borel measurable 
with values in $[0,1]$ as is explained in \cite{Lat03}. As these are
standard results, we omit the details. However, in section 
\ref{sec:cornonsm} below, we will work out in detail a direct 
approximation argument in the setting of Theorem \ref{thm:gbl} that
could also be applied here to deduce the measurable versions of
Corollaries \ref{cor:ehrhard} and \ref{cor:borell}.
\end{rem}

\section{A Gaussian reverse Brascamp-Lieb inequality}
\label{sec:bl}

\subsection{Barthe's inequality}
\label{sec:barthe}

Both the classical Brunn-Minkowski inequality and Ehrhard's inequality 
bound the measure of the Minkowski sum $\lambda A+(1-\lambda)B$ from below 
in terms of the measures of $A$ and $B$. Therefore, when either $A$ or $B$ 
has measure zero, these inequalities necessarily become trivial. 
Nonetheless, it is perfectly possible for $\lambda A+(1-\lambda)B$ to have 
positive measure even when $A$ and $B$ are, for example, contained in 
lower-dimensional subspaces of $\mathbb{R}^n$. This phenomenon is captured 
quantitatively by a significant generalization of the classical 
Brunn-Minkowski inequality due to Barthe \cite{Bar98}, which we presently 
recall.

Fix $\lambda_1,\ldots,\lambda_k\ge 0$, and let $B_1,\ldots,B_k$
be linear maps $B_i:\mathbb{R}^n\to\mathbb{R}^{n_i}$ such that
$$
	\sum_{i=1}^k \lambda_i B_i^*B_i = I_n,\qquad
	B_iB_i^* = I_{n_i}\mbox{ for all }i.
$$
Note that $B_i^*$ isometrically embeds $\mathbb{R}^{n_i}$ in the linear
subspace $E_i=\mathrm{Im}(B_i^*)$ of $\mathbb{R}^n$.
Let $f_i:\mathbb{R}^{n_i}\to\mathbb{R}$ and $h:\mathbb{R}^n\to\mathbb{R}$
be functions such that
$$
	\lambda_1\log (f_1(x_1))+\cdots+
	\lambda_k\log (f_k(x_k)) \le
	\log (h(\lambda_1B_1^*x_1+\cdots+\lambda_kB_k^*x_k))
$$
for all $x_i\in\mathbb{R}^{n_i}$. Then Barthe's inequality states that
$$
	\lambda_1 \log\bigg(\int f_1\,d\gamma_{n_1}\bigg) +
	\cdots +
	\lambda_k \log\bigg(\int f_k\,d\gamma_{n_k}\bigg) \le
	\log\bigg(\int h\,d\gamma_n\bigg)
$$
(see \cite{BH09,Leh14} for the formulation in terms of Gaussian rather 
than Lebesgue measure). When $f_i$ are taken to be indicator 
functions of sets, this reduces to the following generalization of 
the Brunn-Minkowski inequality: for any Borel sets $A_i\subseteq E_i$
$$
	\lambda_1 \log(\gamma_{n_1}(A_1)) + \cdots +
	\lambda_k \log(\gamma_{n_k}(A_k)) \le
	\log(\gamma_n(\lambda_1A_1+\cdots+\lambda_kA_k)),
$$
where we implicitly identify 
$\gamma_{n_i}$ with the standard Gaussian measure on $E_i$.

\begin{rem}
Barthe's inequality is also called the reverse Brascamp-Lieb 
inequality. The classical Brascamp-Lieb inequality is an analogous
multilinear generalization of H\"older's inequality. Just as the
Pr\'ekopa-Leindler inequality could formally be viewed as a reverse form 
of H\"older's inequality, Barthe's inequality can be viewed as a reverse 
form of the Brascamp-Lieb inequality. Let us note that we have stated 
the inequality in its ``geometric'' form, which is most natural for our 
purposes. The general form of the reverse Brascamp-Lieb inequality 
(for general matrices $B_i$) can be deduced from the geometric form, see
\cite{BCCT08} and \cite{Leh14} for details.
\end{rem}

When $n_i=n$ and $B_i=I_n$ for all $i$, Barthe's inequality reduces to the 
Pr\'ekopa-Leindler inequality. However, we know that the latter is far 
from optimal for Gaussian measures: the sharp form of the 
Pr\'ekopa-Leindler inequality in the Gaussian case is precisely Ehrhard's 
inequality (Corollary \ref{cor:ehrhard}), where the logarithm is replaced 
by $\Phi^{-1}$. It is therefore natural to ask whether there exists an 
analogous Gaussian improvement of Barthe's inequality. This question was 
raised in \cite[\S 4.2]{BH09}. We will show in section \ref{sec:gbl} that 
there does in fact exist an interesting family of inequalities of this 
form, but the correct formulation of such inequalities is not entirely 
obvious.  Before we develop these inequalities, let us briefly discuss 
what sort of improvement could reasonably be expected.

One might optimistically hope that as in the case of Ehrhard's inequality, we 
may simply replace $\log$ by $\Phi^{-1}$ in Barthe's inequality to obtain 
the analogus Gaussian form. However, not only is this impossible, but 
in fact no improvement of Barthe's inequality is possible in 
general. To see why, consider the case where $E_1$ and $E_2$ are 
two orthogonal subspaces of $\mathbb{R}^2$, which forces 
$\lambda_1=\lambda_2=1$. Suppose the inequality
$$
	\mathrm{L}(\gamma_{1}(A_1)) +
	\mathrm{L}(\gamma_{1}(A_2)) \le \mathrm{L}(\gamma_2(A_1+A_2))
$$
holds for a function $\mathrm{L}$. As 
$\gamma_2(A_1+A_2)=\gamma_1(A_1)\gamma_1(A_2)$ in this case, we must 
have
$$
	\mathrm{L}(x) + \mathrm{L}(y) \le \mathrm{L}(xy)
$$
for all $x,y\in[0,1]$, which is clearly violated when 
$\mathrm{L}(x)=\Phi^{-1}(x)$ (let $x=y=\frac{1}{2}$). On the other hand,
the above inequality holds with equality when $\mathrm{L}(x)=\log x$. 
It follows that Barthe's inequality is already optimal in the orthogonal 
setting and cannot be improved by any alternative choice of function 
$\mathrm{L}$.

We have now considered two extreme cases. When 
$E_1=\cdots=E_k=\mathbb{R}^n$, Ehrhard's inequality is sharp and the
optimal choice of function is $\mathrm{L}=\Phi^{-1}$. On the other hand, 
when $E_1,\ldots,E_k$ are orthogonal subspaces, Barthe's inequality 
is sharp and the optimal choice of function is $\mathrm{L}=\log$. 
One can therefore not expect that any single choice of function 
$\mathrm{L}$ can provide a systematic Gaussian refinement of Barthe's 
inequality: any general improvement requires the choice of $\mathrm{L}$ to 
depend at least on the parameters $\lambda_i$ and $B_i$. This feature is 
integral to the formulation of the Gaussian reverse Brascamp-Lieb 
inequalities that we will prove presently: we will 
introduce a family of inequalities that interpolate, in some sense, 
between the Ehrhard and Barthe inequalities; the best choice of inequality 
within this family must depend on the parameters to which it is applied.

\subsection{A Gaussian refinement}
\label{sec:gbl}

In the remainder of this section, we place ourselves in the same setting
as in the above formulation of Barthe's inequality: that is, we 
fix $\lambda_1,\ldots,\lambda_k\ge 0$ and let $B_1,\ldots,B_k$
be linear maps $B_i:\mathbb{R}^n\to\mathbb{R}^{n_i}$ such that
$$
	\sum_{i=1}^k \lambda_i B_i^*B_i = I_n,\qquad
	B_iB_i^* = I_{n_i}\mbox{ for all }i.
$$
As before, we define the subspaces $E_i=\mathrm{Im}(B_i^*)$.
We also define the function
$$
	\Phi^{-1}_c(x) := \Phi^{-1}(cx)-\Phi^{-1}(c),\qquad
	x\in[0,1]
$$
for $c\in(0,1)$. We will prove the following Gaussian form of 
Barthe's inequality.

\begin{thm}
\label{thm:gbl}
Let $c\in(0,1)$, and let $f_1,\ldots,f_k,h$ be Borel measurable functions
$f_i:\mathbb{R}^{n_i}\to\mathbb{R}$, $h:\mathbb{R}^n\to\mathbb{R}$ with 
values in $[0,1]$. Suppose that
$$
	\lambda_1\Phi_c^{-1}(f_1(x_1))+\cdots+
	\lambda_k\Phi_c^{-1}(f_k(x_k)) \le
	\Phi_c^{-1}(h(\lambda_1B_1^*x_1+\cdots+
	\lambda_kB_k^*x_k))
$$
for all $x_i\in\mathbb{R}^{n_i}$. Then
$$
	\lambda_1\Phi_c^{-1}\bigg(\int f_1\,d\gamma_{n_1}\bigg) +
	\cdots +
	\lambda_k\Phi_c^{-1}\bigg(\int f_k\,d\gamma_{n_k}\bigg) 
	\le
	\Phi_c^{-1}\bigg(\int h\,d\gamma_n\bigg).
$$
\end{thm}

We immediately deduce the following generalization of Ehrhard's 
inequality.

\begin{cor}
\label{cor:gblbm}
For any $c\in(0,1)$ and Borel sets $A_i\subseteq E_i$, $i=1,\ldots,k$, we have
$$
	\lambda_1\Phi_c^{-1}(\gamma_{n_1}(A_1)) + \cdots + 
	\lambda_k\Phi_c^{-1}(\gamma_{n_k}(A_k)) \le
	\Phi_c^{-1}(\gamma_n(\lambda_1A_1+\cdots+\lambda_kA_k)).
$$
\end{cor}

\begin{proof}
Choose $f_i(x)=\mathbf{1}_{A_i}(B_i^*x)$ and
$h(x)=\mathbf{1}_{\lambda_1A_1+\cdots+\lambda_kA_k}(x)$.
\end{proof}

It is instructive to note that both Ehrhard's inequality and Barthe's 
generalized Brunn-Minkowski inequality arise as limiting cases of 
Corollary \ref{cor:gblbm}.

Let us first recover Barthe's inequality. To this end, recall that
$$
	\Phi(-y) = \int_y^\infty \frac{e^{-z^2/2}}{\sqrt{2\pi}}\,dz
	= (1+o(1)) \frac{e^{-y^2/2}}{y\sqrt{2\pi}}
	\quad\mbox{as }y\to\infty.
$$
A simple computation shows that
$$
	\Phi^{-1}(x)^2 =
	-2\log x - \log\log(1/x) - \log 4\pi + o(1)
	\quad\mbox{as }x\downarrow 0,
$$
so that
$$
	\Phi_c^{-1}(x) =
	\frac{\Phi^{-1}(cx)^2-\Phi^{-1}(c)^2}{\Phi^{-1}(c)+\Phi^{-1}(cx)}
	=
	\frac{-2\log x + o(1)}{\Phi^{-1}(c)+\Phi^{-1}(cx)}
	\quad\mbox{as }c\downarrow 0.
$$
This implies, in particular, that
$$
	\lim_{c\downarrow 0} \Phi_c^{-1}(x)\sqrt{-2\log c} =
	\log x.
$$
Thus Barthe's Brunn-Minkowski inequality is recovered as
$c\downarrow 0$ in Corollary \ref{cor:gblbm}.

On the other hand, to recover Ehrhard's inequality, set
$n_i=n$ and $B_i=I_n$ for all $i$. This forces 
$\lambda_1+\cdots+\lambda_k=1$, so that Corollary \ref{cor:gblbm}
reduces to
$$
	\lambda_1\Phi^{-1}(c\gamma_{n}(A_1)) + \cdots + 
	\lambda_k\Phi^{-1}(c\gamma_{n}(A_k)) \le
	\Phi^{-1}(c\gamma_n(\lambda_1A_1+\cdots+\lambda_kA_k)).
$$
Thus Ehrhard's inequality is recovered as
$c\uparrow 1$ in Corollary \ref{cor:gblbm}.

We have therefore seen that Corollary \ref{cor:gblbm} is never worse than 
Barthe's Brunn-Minkowski inequality, and can be substantially better. For 
general parameters, one has the freedom to optimize over $c$ to obtain the 
best inequality in this family.

\subsection{Proof of Theorem \ref{thm:gbl}: smooth case}

The main idea that is needed in the proof of Theorem \ref{thm:gbl} is 
the following minor extension of Theorem \ref{thm:main}.

\begin{prop}
\label{prop:phic}
Let $f:\mathbb{R}^m\to(-\infty,0]$ be bounded and uniformly continuous,
let $c\in(0,1)$, and let $B:\mathbb{R}^n\to\mathbb{R}^m$ be a linear map 
such that $BB^*=I_m$. Define
\begin{align*}
	J_f^{B,c}[\alpha,\beta] :=
	\mathbf{E}\bigg[
	&\int_0^1 
	e^{-\frac{1}{2}\int_0^t\|\beta_s\|^2ds}
	\langle B^*B\alpha_t,\beta_t\rangle\,dt 
\\
	&+
	e^{-\frac{1}{2}\int_0^1\|\beta_t\|^2dt}
	f\bigg(
	BW_1 + \int_0^1 
	B\alpha_t\,dt +
	\frac{\Phi^{-1}(c)}{2}\int_0^1 B\beta_t\,dt
	\bigg)
	\bigg]
\end{align*}
for $\alpha,\beta\in\mathcal{C}$. Then we have 
$$
	\Phi_c^{-1}\bigg(
	\int \Phi_c(f)\,d\gamma_m
	\bigg) 
	=
	\sup_{\alpha\in\mathcal{S}}\inf_{\beta\in\mathcal{C}}
	J_f^{B,c}[\alpha(\beta),\beta].
$$
\end{prop}

\begin{proof}
We begin by noting that, by the definition of $\Phi_c^{-1}$, we can write
$$
	\Phi_c^{-1}\bigg(
        \int \Phi_c(f)\,d\gamma_m
        \bigg) =
	\Phi^{-1}\bigg(
        \int \Phi(\Phi^{-1}(c)+f\circ B)\,d\gamma_n
        \bigg)-\Phi^{-1}(c),
$$
where we used that $B$ is a projection (so that 
$\gamma_m=\gamma_nB^{-1}$).

Define $g:=\Phi^{-1}(c)+f\circ B$. As $f\le 0$, we have
$g\le \Phi^{-1}(c)$. Following \emph{verbatim} the proof of the upper 
bound of Theorem \ref{thm:main}, we have
$$
	\Phi^{-1}\bigg(
        \int \Phi(g)\,d\gamma_n\bigg)
	\le
	J_g[\alpha^*(\beta)+\tfrac{1}{2}
	\Phi^{-1}(c)\beta,\beta]
$$
for all $\beta\in\mathcal{C}$, where $\alpha^*\in\mathcal{S}$ is a 
strategy of the form
$$
	\alpha^*_t(\beta) =
	(\tfrac{1}{2}
        \Phi^{-1}(c)-v(t,X_t^\beta))\nabla v(t,X_t^\beta),
	\qquad
	v(t,x) = \Phi^{-1}(\mathbf{E}[\Phi(g(W_1-W_t+x))])
$$
for a suitably defined random process $X^\beta$. The crucial observation
at this point is that as $\nabla g(x) = B^*\nabla f(Bx)$, we have
$\nabla v(t,x)\in\mathrm{Im}(B^*)$ for all $t,x$. In particular, the
optimal strategy $\alpha^*$ satisfies 
$B^*B\alpha^*(\beta)=\alpha^*(\beta)$ for every
$\beta\in\mathcal{C}$. Therefore
$$
	\Phi^{-1}\bigg(
        \int \Phi(g)\,d\gamma_n\bigg)
        \le
	\sup_{\alpha\in\mathcal{S}}
	\inf_{\beta\in\mathcal{C}}
        J_g[B^*B\alpha(\beta)+\tfrac{1}{2}
        \Phi^{-1}(c)\beta,\beta].
$$
On the other hand, the corresponding lower bound follows immediately
from Theorem \ref{thm:main} (as strategies of the form
$B^*B\alpha(\beta)+\tfrac{1}{2}\Phi^{-1}(c)\beta$ form a subset
of all possible strategies $\mathcal{S}$). Putting everything together,
we have now shown that
$$
	\Phi_c^{-1}\bigg(
        \int \Phi_c(f)\,d\gamma_m\bigg) =
	\sup_{\alpha\in\mathcal{S}}
        \inf_{\beta\in\mathcal{C}}
        J_g[B^*B\alpha(\beta)+\tfrac{1}{2}
        \Phi^{-1}(c)\beta,\beta]
	- \Phi^{-1}(c).
$$
To complete the proof, it suffices to note that
\begin{align*}
 &       J_g[B^*B\alpha+\tfrac{1}{2}
        \Phi^{-1}(c)\beta,\beta]
	- \Phi^{-1}(c) 
\\ &=
	J_f^{B,c}[\alpha,\beta]
	+
	\Phi^{-1}(c)\,
	\mathbf{E}\bigg[
	\frac{1}{2}
	\int_0^1 e^{-\frac{1}{2}\int_0^t\|\beta_s\|^2ds}
	\|\beta_t\|^2
	\,dt
	+ 
	e^{-\frac{1}{2}\int_0^1\|\beta_t\|^2dt}
	- 1
	\bigg]
\\
&= J_f^{B,c}[\alpha,\beta],
\end{align*}
where we used the fundamental theorem of calculus.
\end{proof}

With Proposition \ref{prop:phic} in hand, we immediately
obtain:

\begin{cor}
\label{cor:gblsmooth}
Theorem \ref{thm:gbl} is valid under the additional assumption that
the functions $f_1,\ldots,f_k,h$ are uniformly continuous with
values in $[\varepsilon,1]$ for some $\varepsilon>0$.
\end{cor}

The proof is identical to that of Corollary \ref{cor:ehrhard}, and we omit 
the details.

\begin{rem}
Corollary \ref{cor:gblbm} can be deduced directly from Corollary 
\ref{cor:gblsmooth} by introducing smooth approximations of the indicator 
functions of the sets $A_1,\ldots,A_k$ and 
$\lambda_1A_1+\cdots+\lambda_kA_k$. Such an argument is given in 
\cite{Bor03}, and can be readily applied in the present setting. We 
therefore do not need the full strength of Theorem \ref{thm:gbl} to deduce 
Corollary \ref{cor:gblbm}. However, while the proof of Theorem 
\ref{thm:gbl} requires a bit more work, it yields a result that is 
potentially of broader utility.
\end{rem}

\subsection{Proof of Theorem \ref{thm:gbl}: general case}
\label{sec:cornonsm}

The important part Theorem \ref{thm:gbl} is already contained in Corollary 
\ref{cor:gblsmooth} above. The remaining arguments in the proof of Theorem 
\ref{thm:gbl} are technical: we must approximate the 
measurable functions $f_1,\ldots,f_k,h$ by uniformly continuous functions 
so that Corollary \ref{cor:gblsmooth} can be applied. The requisite 
approximation arguments are worked out in this section. (Closely related
approximation arguments can also be found in \cite{CM15}.)

We begin by proving Theorem \ref{thm:gbl} in the case that 
$f_1,\ldots,f_k$ and $h$ are upper-semicontinuous. The following lemma 
makes it possible to approximate upper-semicontinuous functions by 
uniformly continuous functions without violating the assumption of Theorem 
\ref{thm:gbl}, so that Corollary \ref{cor:gblsmooth} can be applied.

\begin{lem}
\label{lem:usc}
Let $f_1,\ldots,f_k,h$ be upper-semicontinuous functions
$f_i:\mathbb{R}^{n_i}\to\mathbb{R}$, $h:\mathbb{R}^n\to\mathbb{R}$ with 
values in $[\varepsilon,1]$ for some $\varepsilon>0$.
Let $c\in(0,1)$, and suppose that
$$
	\lambda_1\Phi_c^{-1}(f_1(x_1))+\cdots+
	\lambda_k\Phi_c^{-1}(f_k(x_k)) \le
	\Phi_c^{-1}(h(\lambda_1B_1^*x_1+\cdots+
	\lambda_kB_k^*x_k))
$$
for all $x_i\in\mathbb{R}^{n_i}$. Then there exist for every
$s>0$ uniformly continuous functions $f_1^s,\ldots,f_k^s,h^s$
with values in $[\varepsilon,1]$ such that
$$
	\lambda_1\Phi_c^{-1}(f_1^s(x_1))+\cdots+
	\lambda_k\Phi_c^{-1}(f_k^s(x_k)) \le
	\Phi_c^{-1}(h^s(\lambda_1B_1^*x_1+\cdots+
	\lambda_kB_k^*x_k))
$$
for all $x_i\in\mathbb{R}^{n_i}$, and such that
$f^s_i\to f_i$ and $h^s\to h$ pointwise as $s\downarrow 0$.
\end{lem}

\begin{proof}
Define $f^s_i$ and $h^s$ by the sup-convolutions
\begin{align*}
	\Phi_c^{-1}(f_i^s(x)) &:=
	\sup_{y\in\mathbb{R}^{n_i}}\{
	\Phi_c^{-1}(f_i(y)) - s^{-1}\|x-y\|\},
	\\
	\Phi_c^{-1}(h^s(x)) &:=
	\sup_{y\in\mathbb{R}^{n}}\{
	\Phi_c^{-1}(h(y)) - s^{-1}\|x-y\|\}.
\end{align*}
It is easily seen that $f_i^s,h^s$ take values in $[\varepsilon,1]$, and
that $\Phi_c^{-1}(f_i^s)$ and $\Phi_c^{-1}(h^s)$ are $s^{-1}$-Lipschitz;
thus $f_i^s$ and $h^s$ are certainly uniformly continuous.
We now claim that $h^s\to h$ as $s\downarrow 0$. To see this,
choose for every $s>0$ a point $y_s$ such that
$$
	\Phi_c^{-1}(h^s(x)) \le
	\Phi_c^{-1}(h(y_s)) - s^{-1}\|x-y_s\| + s.
$$
As $h^s\ge\varepsilon$ and $h\le 1$, this evidently implies
$\|x-y_s\| \le s^2-s\Phi_c^{-1}(\varepsilon)$
for all $s$, so that $y_s\to x$ as $s\downarrow 0$. But we can
now estimate
\begin{align*}
	\Phi_c^{-1}(h(x)) &\le
	\liminf_{s\downarrow 0}\Phi_c^{-1}(h^s(x))
	\le
	\limsup_{s\downarrow 0}\Phi_c^{-1}(h^s(x)) \\
	&\le
	\limsup_{s\downarrow 0}
	\Phi_c^{-1}(h(y_s))
	\le
	\Phi_c^{-1}(h(x)),
\end{align*}
where we have used that $h$ is upper-semicontinuous in the last line. This 
shows that $h^s\to h$ pointwise as $s\downarrow 0$, and $f_i^s\to 
f_i$ follows identically. Finally, note that
\begin{align*}
	&\lambda_1\Phi_c^{-1}(f_1^s(x_1))+\cdots+
	\lambda_k\Phi_c^{-1}(f_k^s(x_k)) 
	\\
	&=
	\sup_{y_1,\ldots,y_k}\{
	\lambda_1\Phi_c^{-1}(f_1(y_1))+\cdots+
        \lambda_k\Phi_c^{-1}(f_k(y_k)) \\
	&\qquad\qquad\qquad\qquad -
	s^{-1}\lambda_1\|x_1-y_1\|-\cdots-s^{-1}\lambda_k\|x_k-y_k\|\}
	\\
	&
	\le
	\sup_{y_1,\ldots,y_k}\{
	\Phi_c^{-1}(h(\lambda_1B_1^*y_1+\cdots+
        \lambda_kB_k^*y_k)) \\
	&\qquad\qquad\qquad\qquad
	-s^{-1}\|\lambda_1B_1^*(x_1-y_1)+\cdots+
	\lambda_kB_k^*(x_k-y_k)\|\}
	\\
	&
	\le
	\Phi_c^{-1}(h^s(\lambda_1B_1^*x_1+\cdots+
	\lambda_kB_k^*x_k)),
\end{align*}
where we have used that $\|B_i^*z\|=\|z\|$ for $z\in\mathbb{R}^{n_i}$
and the triangle inequality.
\end{proof}

Using Lemma \ref{lem:usc} and Corollary \ref{cor:gblsmooth},
we can now prove the following.

\begin{cor}
\label{cor:blusc}
Theorem \ref{thm:gbl} is valid under the additional assumption that
the functions $f_1,\ldots,f_k,h$ are upper-semicontinuous with
values in $[0,1]$.
\end{cor}

\begin{proof}
We first approximate $f_1,\ldots,f_k,h$ by functions that are
bounded away from zero. To this end, fix $\varepsilon\in(0,1)$ and let 
$\delta := \max_i \Phi_c(\lambda_i\Phi_c^{-1}(\varepsilon))$. Define
the upper-semicontinuous functions
$\bar h:=h\vee\delta$ and $\bar f_i := f_i\vee\varepsilon$ for all $i$.
We claim that
$$
	\lambda_1\Phi_c^{-1}(\bar f_1(x_1))+\cdots+
	\lambda_k\Phi_c^{-1}(\bar f_k(x_k)) \le
	\Phi_c^{-1}(\bar h(\lambda_1B_1^*x_1+\cdots+
	\lambda_kB_k^*x_k)).
$$
Indeed, if $f_i(x_i)>\varepsilon$ for all $i$ this follows
from the assumption of Theorem \ref{thm:gbl}, while if
$f_i(x_i)\le\varepsilon$ for some $i$ the left-hand side is
at most $\Phi_c^{-1}(\delta)$.

Applying Lemma \ref{lem:usc}, we can find uniformly continuous
functions $\bar f_1^s,\ldots,\bar f_k^s,\bar h^s$ with values in 
$[\varepsilon,1]$ such that $\bar f_i^s\to\bar f_i$ and $\bar h^s\to\bar 
h$ pointwise as $s\downarrow 0$ and
$$
	\lambda_1\Phi_c^{-1}(\bar f_1^s(x_1))+\cdots+
	\lambda_k\Phi_c^{-1}(\bar f_k^s(x_k)) \le
	\Phi_c^{-1}(\bar h^s(\lambda_1B_1^*x_1+\cdots+
	\lambda_kB_k^*x_k))
$$
for every $s>0$. Corollary \ref{cor:gblsmooth} implies
$$
	\lambda_1\Phi_c^{-1}\bigg(\int \bar f_1^s\,d\gamma_{n_1}\bigg) +
	\cdots +
	\lambda_k\Phi_c^{-1}\bigg(\int \bar f_k^s\,d\gamma_{n_k}\bigg) 
	\le
	\Phi_c^{-1}\bigg(\int \bar h^s\,d\gamma_n\bigg).
$$
The conclusion follows using dominated convergence as $s\downarrow 
0$ and $\varepsilon\downarrow 0$.
\end{proof}

We can now complete the proof of Theorem \ref{thm:gbl}.

\begin{proof}[Proof of Theorem \ref{thm:gbl}]
Let $\tilde f_1,\ldots,\tilde f_k$ be 
upper-semicontinuous functions with compact support and with values in 
$[0,1]$ such that $\tilde f_i\le f_i$ for all $i$.
Define $\tilde h$ by
$$
	\Phi^{-1}_c(\tilde h(x)) :=
	\sup_{\lambda_1B_1^*x_1+\cdots+\lambda_kB_k^*x_k=x}
	\{
	\lambda_1\Phi_c^{-1}(\tilde f_1(x_1))+\cdots+
        \lambda_k\Phi_c^{-1}(\tilde f_k(x_k))	
	\}.
$$
Then $\tilde h\le h$ by construction, and $\tilde h$ is also 
upper-semicontinuous \cite[Prop.\ 1.27]{RW98}. Moreover, the 
upper-semicontinuous functions $\tilde f_1,\ldots,\tilde f_k$ and $\tilde 
h$ clearly satisfy the assumptions of Theorem \ref{thm:gbl}. Therefore, 
Corollary \ref{cor:blusc} implies
$$
	\lambda_1\Phi_c^{-1}\bigg(\int \tilde f_1\,d\gamma_{n_1}\bigg) +
	\cdots +
	\lambda_k\Phi_c^{-1}\bigg(\int \tilde f_k\,d\gamma_{n_k}\bigg) 
	\le
	\Phi_c^{-1}\bigg(\int h\,d\gamma_n\bigg).
$$
The conclusion now follows by taking the supremum on the left-hand side
over all compactly supported upper-semicontinuous functions 
$\tilde f_i\le f_i$ \cite[Prop.\ 7.14]{Fol99}.
\end{proof}

\section{Generalized means}
\label{sec:hlp}

Unlike the logarithmic functional $f\mapsto \log(\int e^f d\gamma_n)$, 
whose stochastic representation has a natural interpretation through the 
Gibbs variational principle, the emergence of a stochastic game 
representation for $f\mapsto \Phi^{-1}(\int \Phi(f)\,d\gamma_n)$ may 
appear rather unexpected. To provide some further insight into such 
representations, we aim in this section to place the result of Theorem 
\ref{thm:main} in a broader context.

Throughout this section, let $I\subset\mathbb{R}$ be a compact
interval, and let $F:I\to\mathbb{R}$ be a smooth function 
that is strictly increasing $F'>0$. Following Hardy, Littlewood, and 
P\'olya \cite[chapter 3]{HLP88},
we define the generalized mean $\mathfrak{M}_F$ as
$$
	\mathfrak{M}_F(f) := F^{-1}\bigg(
	\int F(f)\,d\gamma_n
	\bigg)
$$
for any measurable function $f:\mathbb{R}^n\to I$.  We will argue below 
that the generalized mean $\mathfrak{M}_F$ admits a stochastic 
representation for any sufficiently regular function $F$: from 
this perspective, there is nothing particularly special about the 
specific cases $F(x)=e^x$ and $F(x)=\Phi(x)$ that we encountered so far. 
Of course, the potential utility of such stochastic representations in 
other settings depends on the problem at hand. For example, to establish 
Brunn-Minkowski type inequalities, we crucially exploited a special 
feature of the functions $F(x)=e^x$ and $F(x)=\Phi(x)$: in both cases, the 
running cost in the stochastic representation proves to be a concave 
function of the strategy that is being maximized over. While such 
structural features of the representation are specific to particular 
choices of $F$, the existence of a stochastic representation is not 
anything special in its own right.

In their study of generalized means, Hardy, Littlewood, and P\'olya 
\cite[\S 3.16]{HLP88} obtained necessary and sufficient conditions for 
$f\mapsto\mathfrak{M}_F(f)$ to be a convex functional. In section 
\ref{sec:gmconv}, we will show that stochastic representations provide an 
interesting perspective on this characterization: the conditions of Hardy, 
Littlewood, and P\'olya are precisely those that are needed to obtain a 
stochastic representation for $\mathfrak{M}_F$ involving only a supremum 
(as in the case $F(x)=e^x$). In particular, we can state a very general 
expression for the stochastic representation in this setting, despite that 
the Fenchel transform of $\mathfrak{M}_F$ (and therefore the natural 
analogue of the Gibbs variational principle) rarely admits a tractable 
expression. For generalized means that are not convex, we will outline in 
section \ref{sec:gmgame} how one can obtain in this case a stochastic game 
representation of $\mathfrak{M}_F$ under essentially no assumptions on 
$F$. As the explicit expressions that define such games for general $F$ 
do not provide much insight, we do not state a general theorem, but rather 
illustrate by means of an example how easily such representations can be 
obtained in practice.

\subsection{The convex case}
\label{sec:gmconv}

The following result due to Hardy, Littlewood, and P\'olya characterizes 
precisely when the functional $f\mapsto\mathfrak{M}_F(f)$ is convex.

\begin{thm}[{\cite[\S 3.16]{HLP88}}]
\label{thm:hlp}
The generalized mean functional $f\mapsto\mathfrak{M}_F(f)$ is convex
if and only if the function $F$ is convex and the function $F'/F''$ is 
concave.
\end{thm}

\begin{proof}
The following facts are explicitly stated and proved in
\cite[\S 3.16]{HLP88}:
\begin{itemize}
\item Convexity of $F$ is necessary for $\mathfrak{M}_F$ to be convex.
\vskip.1cm
\item If $F$ is strictly convex $F''>0$, then concavity of $F'/F''$
is necessary and sufficient for $\mathfrak{M}_F$ to be convex.
\end{itemize}
For completeness, we spell out what happens when $F$ is convex but
fails to be strictly convex. We should consider two separate cases:
\begin{itemize}
\item
If $F''$ vanishes everywhere in $I$, then $F$ is linear and convexity
of $\mathfrak{M}_F$ is 
trivial (note that in this case $F'/F''\equiv +\infty$ is clearly 
concave).
\vskip.1cm
\item
If $F''$ vanishes at some point but not everywhere in $I$,
then $F'/F''$ must blow up to $+\infty$ near that point as we assumed 
that $F'>0$ and that $F$ is smooth. This implies there is a subinterval
$J\subset I$ on which $F''>0$ but where $F'/F''$ fails to be concave, 
so convexity of $\mathfrak{M}_F$ must fail.
\end{itemize}
We have therefore established all possible cases of Theorem \ref{thm:hlp}.
\end{proof}

When $F(x)=e^x$, the condition of Theorem \ref{thm:hlp} is evidently 
satisfied; in this case, convexity of $\mathfrak{M}_F$ is simply the 
statement of H\"older's inequality. On the other hand, when 
$F(x)=\Phi(x)$, the condition for convexity fails to be satisfied on any 
interval $I$. In particular, while log-concavity could formally be viewed 
as a ``reverse'' form of H\"older's inequality, there cannot exist a 
Gaussian improvement of H\"older's inequality that is analogous to 
Ehrhard's improvement of log-concavity.

\begin{rem}
The proof of Theorem \ref{thm:hlp} shows that unless $F$ is linear,
convexity of the functional $\mathfrak{M}_F$ requires that $F$ is
strictly convex $F''>0$ everywhere in $I$. We will therefore assume the 
latter without loss of generality in our development of stochastic 
representations for convex generalized means.
\end{rem}

The relevance of the conditions of Theorem \ref{thm:hlp} is far from 
obvious at first sight. We will presently see that these conditions arise 
in a very natural manner when we attempt to obtain a stochastic 
representation for $\mathfrak{M}_F$. 

We begin by developing the argument of section \ref{sec:borpde} in the 
present setting. Let $f:\mathbb{R}^n\to I$ be a Lipschitz function and
define for $(t,x)\in[0,1]\times\mathbb{R}^n$
$$
	u(t,x) := \mathbf{E}[F(f(W_1-W_t+x))],
$$
so that $u$ solves the heat equation. Define
$$
	v(t,x) := F^{-1}(u(t,x)).
$$
Note that as $F$ is smooth and $F'>0$, the function $F^{-1}$ is smooth by 
the inverse function theorem. Therefore, by elementary properties of the 
heat equation, $v$ takes values in $I$, is smooth and has bounded 
derivatives of all orders on $[0,1-\varepsilon]\times\mathbb{R}^n$ for 
every $\varepsilon>0$, and $v(t,x)\to f(x)$ uniformly in $x$ as $t\to 1$. 
Using
$$
	\frac{\partial u}{\partial t} = F'(v)\frac{\partial v}{\partial t},
	\qquad
	\Delta u = F'(v)\Delta v + F''(v) \|\nabla v\|^2
$$
and the heat equation for $u$ shows that $v$ satisfies the PDE
$$
	\frac{\partial v}{\partial t}+\frac{1}{2}\Delta v +
	\frac{1}{2}\frac{F''(v)}{F'(v)}\|\nabla v\|^2=0,\qquad
	v(1,x)=f(x).
$$
We now readily see the relevance of the conditions of Theorem 
\ref{thm:hlp}: as the function $(x,y)\mapsto \|x\|^2/y$ is convex
for $(x,y)\in\mathbb{R}^n\times\mathbb{R}_+$, the conditions of Theorem
\ref{thm:hlp} are precisely those that ensure that the nonlinear term in 
this PDE is a convex function of $(\nabla v,v)$. In particular, we can 
express this term as follows.

\begin{lem}
\label{lem:hjb}
Suppose that $F''>0$ and that $F'/F''$ is concave. Denote by
$R:=(-F'/F'')^*$ the Fenchel transform of the convex function $-F'/F''$.
Then
$$
	\frac{1}{2}\frac{F''(v)}{F'(v)}\|\nabla v\|^2 =
	\sup_{a\in\mathbb{R}^n}
	\sup_{b\in\mathbb{R}}
	\bigg\{
	\langle a,\nabla v\rangle +
	\frac{1}{2}v b\|a\|^2
	- \frac{1}{2}R(b)\|a\|^2
	\bigg\},
$$
where the optimizer is $a^*= (F''(v)/F'(v))\nabla v$ and $b^*=
F'(v)F'''(v)/F''(v)^2 - 1$.
\end{lem}

\begin{proof}
The optimization $\sup_{b\in\mathbb{R}}\{vb-R(b)\}=-F'(v)/F''(v)$
is simply the definition of the Fenchel conjugate. Moreover, as $F$ is 
assumed to be smooth, the optimizer is given by $b^*=(-F'/F'')'(v)$ 
\cite[Prop.\ 11.3]{RW98}. The optimization over $a$ is trivial.
\end{proof}

Lemma \ref{lem:hjb} reveals that the partial differential equation 
satisfied by $v$ is none other than the Bellman equation for the
value of a stochastic control problem \cite{FS06}.

\begin{thm}
\label{thm:hlpstoch}
Let $F:I\to\mathbb{R}$ be a nonlinear smooth and strictly increasing 
function such that $\mathfrak{M}_F$ is convex. Then $\mathfrak{M}_F$
admits the stochastic representation
$$
	\mathfrak{M}_F(f) =         
	\sup_{\alpha\in\mathcal{C}_b^n}
        \sup_{\beta\in\mathcal{C}_b^1}
	K_f[\alpha,\beta]
$$
for every lower-semicontinuous function $f:\mathbb{R}^n\to I$, where
$$
	K_f[\alpha,\beta]:=
	\mathbf{E}\bigg[
	e^{\frac{1}{2}\int_0^1
        \beta_t\|\alpha_t\|^2dt}
	f\bigg(
	W_1 + \int_0^1 \alpha_t\,dt
	\bigg)
	-\frac{1}{2}\int_0^1 e^{\frac{1}{2}\int_0^t
	\beta_s\|\alpha_s\|^2ds}
	R(\beta_t)\|\alpha_t\|^2 dt
	\bigg]
$$
with $R:=(-F'/F'')^*$. Here $\mathcal{C}_b^k$ denotes the family
of all $k$-dimensional uniformly bounded and progressively measurable 
processes.
\end{thm}

This result should be viewed as an explicit stochastic representation
of Fenchel duality for the convex functional $\mathfrak{M}_F$. In 
particular, as $K_f[\alpha,\beta]$ is linear in $f$, the convexity
of $\mathfrak{M}_F$ is immediately obvious from the representation.

\begin{proof}
We first assume that the function $f$ is Lipschitz.
Define $W_t^\alpha := W_t + \int_0^t \alpha_s\,ds$.
Applying It\^o's formula to $e^{\frac{1}{2}\int_0^t
\beta_s\|\alpha_s\|^2ds}v(t,W_t^\alpha)$ yields
\begin{align*}
	&e^{\frac{1}{2}\int_0^1
	\beta_t\|\alpha_t\|^2dt}f(W_1^\alpha)
	-
	\frac{1}{2}\int_0^1 e^{\frac{1}{2}\int_0^t
	\beta_s\|\alpha_s\|^2ds}
	R(\beta_t)\|\alpha_t\|^2 dt = \\
	&
	v(0,0) +
	\int_0^1 e^{\frac{1}{2}\int_0^t
        \beta_s\|\alpha_s\|^2ds}
	\langle \nabla v(t,W_t^\alpha),dW_t\rangle
	\\
	&
	+
	\int_0^1
	e^{\frac{1}{2}\int_0^t
        \beta_s\|\alpha_s\|^2ds}
	\bigg\{
	\frac{\partial v}{\partial t}(t,W_t^\alpha) +
	\frac{1}{2}\Delta v(t,W_t^\alpha) +
	\langle \alpha_t,\nabla v(t,W_t^\alpha)\rangle \\
	&
	\phantom{+\int_0^1 e^{\frac{1}{2}\int_0^t\beta_s\|\alpha_s\|^2ds}\bigg\{}
	+
	\frac{1}{2} \beta_t\|\alpha_t\|^2
	v(t,W_t^\alpha)
	- \frac{1}{2} R(\beta_t)\|\alpha_t\|^2
	\bigg\}dt.
\end{align*}
As $f$ is Lipschitz, $\nabla v$ is uniformly bounded so the
stochastic integral is a martingale.
Taking the expectation, and using Lemma \ref{lem:hjb} and the partial 
differential equation for $v$ yields $K_f[\alpha,\beta] \le v(0,0) = 
\mathfrak{M}_F(f)$ for every $\alpha\in\mathcal{C}_b^n$ and 
$\beta\in\mathcal{C}_b^1$. Thus
$$
	\sup_{\alpha\in\mathcal{C}_b^n}
        \sup_{\beta\in\mathcal{C}_b^1}
	K_f[\alpha,\beta] \le \mathfrak{M}_F(f).	
$$
It remains to note that the inequality is equality if we choose
the optimal controls
\begin{align*}
	\alpha^*_t &=
	\frac{F''(v(t,X_t))}{F'(v(t,X_t))}\nabla v(t,X_t),\\
	\beta^*_t &= 
	\frac{F'(v(t,X_t))F'''(v(t,X_t))}{F''(v(t,X_t))^2} - 1,
\end{align*}
where $X_t$ is the solution of the stochastic differential equation
$$
	dX_t = \frac{F''(v(t,X_t))}{
	F'(v(t,X_t))}\nabla v(t,X_t)\,dt+ dW_t,
	\qquad X_0=0.
$$
Here we note that by our assumptions, $F$ has bounded derivatives
of all orders and $F'$ and $F''$ are uniformly
bounded away from zero, $v$ and $\nabla v$ are uniformly bounded,
and $v(t,\cdot)$ has bounded derivatives of all orders for $t<1$,
so that this stochastic differential equation
has a unique strong solution and $\alpha^*\in\mathcal{C}_b^n$,
$\beta^*\in\mathcal{C}_b^1$.

Now assume $f$ is only lower-semicontinuous. 
Let $f_k(x) = \inf_y\{f(y)+k\|x-y\|\}$. Then $f_k:\mathbb{R}^n\to I$
is Lipschitz for every $k$ and $f_k\uparrow f$ pointwise as in the
proof of Lemma \ref{lem:usc}. The result follows using monotone 
convergence by applying the stochastic representation of 
$\mathfrak{M}_F(f_k)$ and taking the supremum over $k$.
\end{proof}

Let us illustrate Theorem \ref{thm:hlpstoch} in some simple examples.

\begin{example}
Consider the case $F(x)=e^x$ that arises from the Gibbs variational
principle. Then $F'/F''\equiv 1$, so we readily compute 
$R=(-F'/F'')^*$ as
$$
	R(x) = \left\{
	\begin{array}{ll}
	1 &\mbox{for }x=0,\\
	+\infty & \mbox{for }x\ne 0.
	\end{array}
	\right.
$$
Substituting this expression into Theorem \ref{thm:hlpstoch}, we 
immediately recover the stochastic representation discussed in the 
introduction.
\end{example}

\begin{example}
Consider the case $F(x)=x^p$ for $p>1$, where we choose $I=[x_-,x_+]$
for some $0<x_-<x_+<\infty$. Then $F',F''>0$ and $F'(x)/F''(x)= x/(p-1)$
is certainly concave. We readily compute $R=(-F'/F'')^*$ as
$$
	R(x) = \left\{
	\begin{array}{ll}
	0 &\mbox{for }x=-1/(p-1),\\
	+\infty & \mbox{for }x\ne -1/(p-1).
	\end{array}
	\right.	
$$
Substituting this expression into Theorem \ref{thm:hlpstoch} yields
$$
	\bigg(\int f^p d\gamma_n\bigg)^{1/p} =
	\sup_{\alpha\in\mathcal{C}_b^n}
	\mathbf{E}\bigg[
	e^{-\frac{1}{2(p-1)}\int_0^1
        \|\alpha_t\|^2dt}
	f\bigg(
	W_1 + \int_0^1 \alpha_t\,dt
	\bigg)\bigg].
$$
This result could also be obtained along the lines of
\cite{BD98} by applying Girsanov's theorem to the representation
$(\int f^p d\gamma_n)^{1/p} = \sup_{g>0:\int g\, d\gamma_n=1}
\int g^{1-1/p}f\,d\gamma_n$.
\end{example}

\begin{example}
Consider the case $F(x)=xe^x$ on $I=[0,C]$ for some $C<\infty$.
Then $F',F''>0$ and $F'(x)/F''(x) = (1+x)/(2+x)$ is concave.
We compute
$$
	R(x) = \left\{
	\begin{array}{ll}
	-2\sqrt{-x}-2x+1 &\mbox{for }x\le 0,\\
	+\infty & \mbox{for }x>0.
	\end{array}
	\right.	
$$
Substituting this expression into Theorem \ref{thm:hlpstoch} yields
\begin{align*}
	&\mathrm{W}\bigg(\int fe^f\,d\gamma_n\bigg) =
	\sup_{\alpha\in\mathcal{C}_b^n}\sup_{\gamma\in\mathcal{C}_b^1}
	\mathbf{E}\bigg[
	e^{-\frac{1}{2}\int_0^1
        \gamma_t^2\|\alpha_t\|^2dt}
	f\bigg(
	W_1 + \int_0^1 \alpha_t\,dt
	\bigg) \\
	& 
\phantom{\mathrm{W}\bigg(\int fe^f\,d\gamma_n\bigg) =
        \sup_{\alpha\in\mathcal{C}_b^n}\sup_{\gamma\in\mathcal{C}_b^1}
\mathbf{E}\bigg[}
	-\frac{1}{2}\int_0^1 e^{-\frac{1}{2}\int_0^t
	\gamma_s^2\|\alpha_s\|^2ds}
	(2\gamma_t^2-2|\gamma_t|+1)\|\alpha_t\|^2 dt
	\bigg],
\end{align*}
where $\mathrm{W}$ is the Lambert $W$-function and
we defined $\beta_t:=-\gamma_t^2$ to enforce nonpositivity.
This expression can be simplified slightly by introducing the
new control $\eta_t := |\gamma_t|\alpha_t$. Rearranging the above 
expression then yields
\begin{align*}
	&\mathrm{W}\bigg(\int fe^f\,d\gamma_n\bigg) =
	\sup_{\alpha,\eta\in\mathcal{C}_b^n}
	\mathbf{E}\bigg[
	e^{-\frac{1}{2}\int_0^1\|\eta_t\|^2dt}
	f\bigg(
	W_1 + \int_0^1 \alpha_t\,dt
	\bigg) \\
	& 
\phantom{\mathrm{W}\bigg(\int fe^f\,d\gamma_n\bigg) =
        \sup_{\alpha,\eta\in\mathcal{C}_b^n}
\mathbf{E}\bigg[}
	-\frac{1}{2}\int_0^1 e^{-\frac{1}{2}\int_0^t\|\eta_s\|^2ds}
	(\|\eta_t\|^2+\|\alpha_t-\eta_t\|^2)dt
	\bigg].
\end{align*}
We remark that the stochastic representation appears in surprisingly
tractable form, while it is not clear whether it is possible to obtain 
a tractable analogue of the Gibbs variational principle
$\mathfrak{M}_F(f) = \sup_{\mu}\{\int f\,d\mu-\mathfrak{M}_F^*(\mu)\}$
in this example.
\end{example}

\subsection{Generalized means and stochastic games}
\label{sec:gmgame}

The essential idea behind the proof of Theorem \ref{thm:hlpstoch}
was that when $\mathfrak{M}_F$ is convex, the nonlinear equation
$$
	\frac{\partial v}{\partial t}+\frac{1}{2}\Delta v +
	\frac{1}{2}\frac{F''(v)}{F'(v)}\|\nabla v\|^2=0
$$
could be expressed as a supremum of linear parabolic equations
$$
	\sup_a\bigg\{
	\frac{\partial v}{\partial t}+\frac{1}{2}\Delta v +
	\langle c_1(a),\nabla v\rangle +
	c_2(a)v + c_3(a)
	\bigg\}=0
$$
for some functions $c_1,c_2,c_3$. Such a representation cannot
hold when $\mathfrak{M}_F$ fails to be convex. Nonetheless, even
in the absence of convexity, we can try to express the above nonlinear 
equation in the more complicated form
$$
	\sup_a\inf_b\bigg\{
	\frac{\partial v}{\partial t}+\frac{1}{2}\Delta v +
	\langle c_1(a,b),\nabla v\rangle +
	c_2(a,b)v + c_3(a,b)
	\bigg\}=0
$$
for some functions $c_1,c_2,c_3$. If this is possible, then the arguments 
of Theorem \ref{thm:main} could be adapted to obtain a stochastic game 
representation for $\mathfrak{M}_F$.

It has long been understood in the PDE literature that while convexity is 
a very special property, almost any reasonable nonlinearity can be 
expressed in the form of a game; see \cite{Ev07,ES84} and the references 
therein. In the present context, this implies that it is possible to 
obtain stochastic game representations for generalized means 
$\mathfrak{M}_F$ under essentially no assumptions on the function $F$.
Let us outline one particular approach to obtaining such representations.
In \cite[\S 4.1]{Ev07}, it is observed that any locally Lipschitz function
$\Psi:\mathbb{R}^k\to\mathbb{R}$ can be represented as
$$
	\Psi(x) = 
	\max_{a\in\mathbb{R}^k}\min_{b\in\mathbb{R}^k}
	\bigg\{
	\int_0^1 \langle\nabla\Psi((1-t)a+tb),x-a\rangle\,dt
	+\Psi(a)
	\bigg\}
$$
(indeed, it suffices to note that $a^*=b^*=x$ is a saddle point).
Now assume, as in the general setting of this section, that
$F:I\to\mathbb{R}$ is a smooth function on the compact interval $I$
that is strictly increasing $F'>0$, and consider the function
$$
	\Psi(x,y) =
	\frac{1}{2}\frac{F''(x)}{F'(x)}\|y\|^2,\qquad
	(x,y)\in I\times\mathbb{R}^n.
$$
Then the gradient of $\Psi$ is locally bounded, and thus the
above maximin representation holds. It is a simple exercise to
repeat the proof of Theorem \ref{thm:main} in the present setting
to obtain a completely general stochastic game representation for
$\mathfrak{M}_F$.

\begin{rem}
There are two minor issues that require care in extending
Theorem \ref{thm:main} to the general setting. First, the representation
will hold for functions $f$ that are smooth with bounded derivatives, but 
one cannot trivially extend to bounded uniformly continuous
functions (as in the present case the exponential factor in front of $f$
in the representation need not have a universal upper bound).
Second, for the same reason, one should work with the smaller
classes of controls and strategies $\mathcal{C}_b:=\{\beta\in\mathcal{C}:
\|\beta\|_\infty<\infty\}$,
$\mathcal{S}_b:=\{\alpha\in\mathcal{S}:\sup_{\|\beta\|_\infty\le R}
\|\alpha(\beta)\|_\infty<\infty~\forall\, R<\infty\}$.
\end{rem}

Carrying out the approach outlined above would give rise to a very general 
representation of $\mathfrak{M}_F$ as a stochastic game. However, this 
representation is not canonical. The usefulness of a stochastic game 
representation in a given situation will generally rely on some structural 
properties of the representation that may be far from evident in this 
particular formulation. For example, applying the above representation to 
the case $F(x)=\Phi(x)$ yields a rather ugly expression from which one 
would be hard-pressed to conclude the validity of Ehrhard's inequality.
While the existence of stochastic game representations for general
$\mathfrak{M}_F$ sheds some light on the origin of the phenomenon observed
in Theorem \ref{thm:main}, a genuinely useful representation of this
kind should be specifically chosen to possess the desired structural 
properties that are relevant to the problem under consideration.
We have already seen an example of this in the previous section, where
special representations were chosen for convex $\mathfrak{M}_F$ from
which the convexity property becomes evident, and in Theorem 
\ref{thm:main}. As a further illustration we provide one additional 
example.

\begin{example}
Consider $F(x)=1-e^{-x^2/2}$ on $I=[\varepsilon,2c]$ for 
$0<\varepsilon<2c<\infty$. This function behaves very similarly to 
$\Phi(x)$ as $x\to\infty$, at least to leading order in the exponent.
We might therefore expect a stochastic game representation of
$\mathfrak{M}_F$ that is similar to that of Theorem \ref{thm:main}.
Let us see how this can be achieved.

We begin by computing
$$
	\frac{F''(x)}{F'(x)} = -x+\frac{1}{x}.
$$
Note that the term $-x$ is precisely what arises for $\Phi$, but
we now have an additional term. To obtain a representation
that is similar to that for $\Phi$, we apply Lemma 
\ref{lem:isaacs} to the first term and introduce an additional control
for the second term:
\begin{align*}
	&\frac{1}{2}\frac{F''(v)}{F'(v)}\|\nabla v\|^2 =
	-\frac{1}{2}v\|\nabla v\|^2 +
	\frac{1}{2v}\|\nabla v\|^2 \\
	&=
	\sup_{a\in\mathbb{R}^n}\inf_{b\in\mathbb{R}^n}
	\bigg\{
	\langle a+cb,\nabla v+b\rangle
	-\frac{1}{2}v\|b\|^2
	\bigg\} +
	\sup_{\tilde a\in\mathbb{R}^n}
	\bigg\{\langle\tilde a,\nabla v\rangle -
	\frac{1}{2}v\|\tilde a\|^2\bigg\}.
\end{align*}
One can now repeat the proof of Theorem \ref{thm:main}
to obtain the representation
\begin{align*}
	\mathfrak{M}_F(f) &=
	\sqrt{-2\log\bigg(\int e^{-f^2/2}\,d\gamma_n\bigg)}
	\\
	&= \sup_{\alpha,\tilde\alpha\in\mathcal{S}}
	\inf_{\beta\in\mathcal{C}}
	\mathbf{E}\bigg[
	\int_0^1 
	e^{-\frac{1}{2}\int_0^t(\|\tilde\alpha_s(\beta)\|^2+\|\beta_s\|^2)ds}
	\langle \alpha_t(\beta),\beta_t\rangle\,dt \\
	&
	\phantom{= \sup_{\alpha,\tilde\alpha\in\mathcal{S}}
        \inf_{\beta\in\mathcal{C}}\mathbf{E}\bigg[}
	+
	e^{-\frac{1}{2}\int_0^1(\|\tilde\alpha_t(\beta)\|^2+\|\beta_t\|^2)dt}
	f\bigg(
	W_1 + \int_0^1 (\alpha_t(\beta)+\tilde\alpha_t(\beta))\,dt
	\bigg)
	\bigg]
\end{align*}
for any bounded, uniformly continuous, and nonnegative function $f$.
Notice that, while this representation is quite close to that of Theorem
\ref{thm:main} (in particular, we see that $\mathfrak{M}_F\ge
\mathfrak{M}_\Phi$ by setting $\tilde\alpha=0$), the present 
representation is not concave in $\tilde\alpha$ and we therefore do not 
obtain an Ehrhard-type inequality for $\mathfrak{M}_F$.
\end{example}

\def\polhk#1{\setbox0=\hbox{#1}{\ooalign{\hidewidth
  \lower1.5ex\hbox{`}\hidewidth\crcr\unhbox0}}}

\end{document}